\documentclass{amsart}[10pt]

\usepackage{amsmath,mathptmx}
\usepackage{amssymb}
\usepackage{graphicx}
\usepackage{tikz, tikz-cd}
\usepackage{color}
\usepackage[all]{xy}
\usepackage{amsthm}
\usepackage{verbatim}
\usepackage{mathrsfs}
\usepackage{float}
\usepackage{lineno}

\usetikzlibrary{arrows,shapes,positioning}
\usetikzlibrary{decorations.markings}
\tikzstyle arrowstyle=[scale=1]
\tikzstyle directed=[postaction={decorate,decoration={markings,
    mark=at position .55 with {\arrow[arrowstyle]{stealth}}}}]
\tikzstyle ddirected=[postaction={decorate,decoration={markings,
    mark=at position .45 with {\arrow[arrowstyle]{stealth}},
    mark=at position .55 with {\arrow[arrowstyle]{stealth}}}}]
\tikzstyle reverse directed=[postaction={decorate,decoration={markings,
    mark=at position .55 with {\arrowreversed[arrowstyle]{stealth};}}}]
\tikzstyle reverse ddirected=[postaction={decorate,decoration={markings,
    mark=at position .55 with {\arrowreversed[arrowstyle]{stealth};},
    mark=at position .65 with {\arrowreversed[arrowstyle]{stealth};}}}]

\usepackage[colorlinks,linkcolor=blue,citecolor=blue,pdfstartview=FitH]{hyperref}

\theoremstyle{plain}

\newtheorem{para}{}[section]
\newtheorem{thm}[para]{Theorem}
\newtheorem{prop}[para]{Proposition}
\newtheorem{lemma}[para]{Lemma}
\newtheorem{cor}[para]{Corollary}

\newtheorem*{thm1}{Theorem}

\theoremstyle{remark}
\newtheorem*{claim}{Claim}
\newtheorem*{remark}{Remark}

\theoremstyle{definition}
\newtheorem{dfn}{Definition}
\newtheorem{example}[para]{Example}

\newcommand{\co}{\colon\thinspace}

\newcommand{\C}{\mathbb{C}}
\newcommand{\Q}{\mathbb{Q}}
\newcommand{\Z}{\mathbb{Z}}

\newcommand{\calb}{\mathcal{B}}
\newcommand{\calc}{\mathcal{C}}
\newcommand{\calh}{\mathcal{H}}

\newcommand\be{\mathbf{e}}
\newcommand\bp{\mathbf{p}}

\newcommand\bv{\mathbf{v}}

\newcommand\bx{\mathbf{x}}
\newcommand\by{\mathbf{y}}

\newcommand{\innp}[1]{\left< #1 \right>}

\newcommand{\norm}[1]{\left\vert \left\vert #1\right\vert\right\vert}

\DeclareMathOperator{\SL}{SL} \DeclareMathOperator{\PSL}{PSL}
\DeclareMathOperator{\GL}{GL} 
 
 \DeclareMathOperator{\SO}{SO}

 \DeclareMathOperator{\Mat}{Mat}

\DeclareMathOperator{\lcm}{lcm}
\DeclareMathOperator{\Gal}{Gal}
\DeclareMathOperator{\Nr}{\mathrm{Nrm}}

\DeclareMathOperator{\vol}{vol}
\DeclareMathOperator{\Id}{Id}

\DeclareMathOperator{\denom}{\mathrm{denom}}
\DeclareMathOperator{\R}{\mathbb{R}}\DeclareMathOperator{\Ram}{\textrm{Ram}}
\DeclareMathOperator{\disc}{disc}\DeclareMathOperator{\cvol}{\mathrm{covol}}

\begin{document}

\title[Virtual and residual properties]{Effective virtual and residual properties \\  of some arithmetic hyperbolic $3$--manifolds}
\author[J. DeBlois]{Jason DeBlois}
\address{Department of Mathematics\\University of Pittsburgh\\Pittsburgh, PA 15260}
\email{jdeblois@pitt.edu}
\author[N. Miller]{Nicholas Miller}
\address{Department of Mathematics\\Indiana University\\Bloomington, IN 47405}
\email{nimimill@iu.edu}
\author[P. Patel]{Priyam Patel}
\address{Department of Mathematics\\University of Utah\\Salt Lake City, UT 84111}
\email{patelp@math.utah.edu}

\begin{abstract}  We give an effective upper bound, for certain arithmetic hyperbolic 3--manifold groups obtained from a quadratic form construction, on the minimal index of a subgroup that embeds in a fixed 6--dimensional right-angled reflection group, stabilizing a totally geodesic subspace. In particular, for manifold groups in any fixed commensurability class we show that the index of such a subgroup is asymptotically smaller than any fractional power of the volume of the manifold. We also give effective bounds on the geodesic residual finiteness growths of closed hyperbolic manifolds that totally geodesically immerse in non-compact right-angled reflection orbifolds, extending work of the third author from the compact case. The first result gives examples to which the second applies, and for these we give explicit bounds on geodesic residual finiteness growth.\end{abstract}
\maketitle

Mal'cev \cite{Malcev} proved that a finitely generated linear group $G$ is \textit{residually finite}: for any non-identity element $g\in G$ there is a finite-index subgroup $H\leq G$ such that $g\notin H$. It follows that the fundamental groups of finite-volume hyperbolic manifolds are residually finite. This property and generalizations such as locally extended residually finite (LERF) have emerged as important tools in the topological study of hyperbolic manifolds, see eg.~\cite{Scott}.  Residual properties are closely tied to virtual ones --- those possessed by covers of finite degree, or, at the level of fundamental group, subgroups of finite index. Recent work of Agol \cite{Agol_vHaken} simultaneously proves the ``virtual conjectures'' \cite[Qns 16--18]{Thurston} for all hyperbolic $3$--manifolds and establishes that their fundamental groups are LERF \cite[Qn~15]{Thurston}.

The full resolution of these conjectures relies on various results, including the surface subgroup theorem \cite{Kahn_Mark}, Canary's covering theorem \cite{Canary_covering}, the tameness theorem \cite{Agol_tameness, CG_tameness}, and Agol's RFRS condition \cite{Agol_RFRS}, to name a few. Building upon the work of Wise and his collaborators initiated in \cite{HagWi_special}, Agol proves the final crucial component in \cite{Agol_vHaken}, showing that the fundamental groups of hyperbolic 3--manifolds are \emph{virtually special}. A group is called \emph{special} if it embeds in a right-angled Coxeter group (\textit{$\mathrm{C}$-special}) or a right-angled Artin group (\textit{$\mathrm{A}$-special}), and is called \emph{virtually special} if it has a finite index special subgroup. Every right-angled Artin group is a finite index subgroup of some right-angled Coxeter group \cite{Davis_Jan} and so every virtually special group is virtually C-special. In this paper, we will focus on C-specialness. 

For a virtually special group, the extensive combinatorial machinery for Coxeter and Artin groups may be brought to bear to establish many desirable properties of its finite-index special subgroups. For example, an important implication of a hyperbolic $3$--manifold group $\pi_1 M$ being virtually C-special is that the finite degree cover of the manifold corresponding to the special subgroup is in fact Haken. However, the virtually special machine does not currently offer an \textit{effective} means for bounding the index of a special subgroup of $\pi_1 M$.  Our first main result provides such a bound for a class of arithmetic hyperbolic $3$--manifolds, and consequently quantifies the virtually Haken property for these manifolds.  This class of groups contains another class for which the LERF property has been known, by work of Agol--Long--Reid \cite{ALR}, for almost two decades.

We now briefly describe this class. We refer the reader to \S 2 for any requisite background material and terminology used in the introduction. Suppose $q$ is a symmetric bilinear form with coefficients in $\mathbb{Q}$, of signature $(3,1)$. Then there is  $P \in \GL(4,\mathbb{R})$ with coefficients in $\R$ such that $P q P^t$ is the diagonal form $\innp{1,1,1,-1}$.  The matrix $P$ conjugates $\mathrm{SO}^+(q,\mathbb{R})$ to the orientation preserving isometry group $\mathrm{Isom}^+(\mathbb{H}^3)$ of $\mathbb{H}^3$ in the hyperboloid model, and conjugates $\mathrm{SO}^+(q,\mathbb{Z})$ to a lattice subgroup of $\mathrm{Isom}^+(\mathbb{H}^3)$.  We will refer to any group of hyperbolic isometries that shares a finite-index subgroup with this conjugate as an arithmetic lattice commensurable with $\mathrm{SO}^+(q,\mathbb{Z})$.

Given a right-angled polyhedron $P\subset\mathbb{H}^n$, let $\Gamma_P$ denote the group of isometries generated by reflections in the sides of $P$; $\Gamma_P$ is a right-angled Coxeter group. Below we are particularly interested in the polyhedron $P_6$. Recall that $\mathrm{SO}^+(6,1;\mathbb{Z})$ is the group generated by reflections in the sides of a simplex $\sigma$ in $\mathbb{H}^6$ that has one ideal vertex, which is itself a fundamental domain for the symmetry group of the right-angled ideal polyhedron $P_6\subset\mathbb{H}^6$ (the Coxeter diagram for $\sigma$ is given in Figure~\ref{coxeter}).  In fact, $P_6$ is the union of the translates of $\sigma$ by the spherical reflection group, of order $2^7 3^4 5$, generated by reflections in its sides containing a finite vertex.  See Lemma 3.4 of \cite{ALR} and its proof.
Then, we have the following theorem.

\newcommand\EmbeddingProp{Let $\Gamma$ be an arithmetic lattice that is commensurable with $\SO^+(q,\Z)$ for some $\Q$--defined bilinear form $q$ of signature $(3,1)$. Then for any $\epsilon>0$, there exist constants $C_\epsilon$ and $D$, where $C_\epsilon$ depends only on $\epsilon$ and the commensurability class of $\Gamma$ and $D$ depends only on the commensurability class of $\Gamma$, such that $\Gamma$ has a subgroup $\Delta$ of index at most $C_\epsilon~\! D~\! \mathrm{covol}(\Gamma)^{\epsilon}$ and an injective homomorphism from $\Delta$ to a subgroup of $\SO^+(6,1;\Z)$ that stabilizes a four dimensional time-like subspace of $\mathbb{R}^{6,1}$.

Thus for an arithmetic hyperbolic manifold $M = \mathbb{H}^3/\Gamma$, where $\Gamma$ is such a lattice, there is a cover $\tilde{M}\to M$ of degree at most $(2^73^45~\!C_\epsilon~\! D)\mathrm{vol}(M)^{\epsilon}$ with a totally geodesic immersion to $\mathbb{H}^6/\Gamma_{P_6}$, for a right-angled polyhedron $P_6$.}
\theoremstyle{plain}
\newtheorem*{embedding_prop}{Theorem \ref{Sec:ArEx:T1}}
\begin{embedding_prop}\EmbeddingProp\end{embedding_prop}

%

In the course of establishing Theorem \ref{Sec:ArEx:T1}, we make effective the strategy exploited by Agol--Long--Reid in \cite{ALR}.  Its main idea is to take the direct sum of the quadratic form $q$ associated to $\Gamma$ with a carefully chosen complimentary form over $\mathbb{R}^3$, producing a 7--dimensional form which is conjugate over $\mathbb{Q}$ to the standard form of signature $(6,1)$ (see Subsection \ref{sec:quadconj} for details).

The non-compact manifold groups covered by Theorem \ref{Sec:ArEx:T1} are precisely those commensurable with the Bianchi groups $\mathrm{PSL}(2,\mathcal{O}_d)$, and for each such group the special subgroup we produce lies in its intersection with $\mathrm{PSL}(2,\mathcal{O}_d)$. Here we have switched to the upper half-space model for $\mathbb{H}^3$ and its orientation preserving isometry group $\mathrm{PSL}(2,\mathbb{C})$. Special subgroups of Bianchi groups were produced recently by Chu \cite{Chu}, and some of our results overlap with hers.  In comparing Theorem \ref{Sec:ArEx:T1} with the main result of \cite{Chu}, it is first important to note that Chu's result provides bounds which are both uniform over all $d$ (ours are not) and, for any particular $d$, are several orders of magnitude smaller than the bounds we produce.

However, Theorem \ref{Sec:ArEx:T1} has the benefit of applying to the entire commensurability class of $\mathrm{PSL}(2,\mathcal{O}_d)$ as opposed to just its finite index subgroups.
This results in the addition of a term depending on the commensurability class as well as a term involving volume.
The latter dependence is necessary from the naive observation that each commensurability class of arithmetic Kleinian groups has infinitely many non-conjugate maximal arithmetic lattices $\Gamma_i$ (whose volumes $V_i$ tend to infinity) and therefore the index of $\Gamma_i\cap\PSL(2,\mathcal{O}_d)$ in $\PSL(2,\mathcal{O}_d)$ must tend to infinity as well.
Granting this, we can achieve growth slower than any fractional power of volume asymptotically, which is the best one can hope for using our methods. 
It is possible that a completely different method can remove the dependence on either of these quantities but we do not take up that matter presently.

It is worth mentioning that the confluence of Theorem \ref{prop:deltaconstruct} and \cite[Thm~1.2]{Chu} imply the following:

\newtheorem*{larry}{Corollary \ref{Larry}}
\newcommand\LarryCor{For each square-free $d\in\mathbb{N}$ there is an effectively computable constant $C_\epsilon = C(\epsilon,d)$ such that for any lattice $\Gamma\subset\mathrm{PSL}(2,\mathbb{C})$ that is commensurable with $\mathrm{PSL}(2,\mathcal{O}_d)$, where $\mathcal{O}_d$ is the ring of integers of $\mathbb{Q}(\sqrt{-d})$, $\Gamma$ has a special subgroup $\Delta$ of index at most $120\,C_\epsilon\,\mathrm{covol}(\Gamma)^\epsilon$.}
\begin{larry}\LarryCor\end{larry}

Theorem \ref{Sec:ArEx:T1} also covers a wider class of compact arithmetic hyperbolic $3$--manifolds.  Section 5 of \cite{Chu} covers the compact arithmetic manifolds associated to $\mathrm{SO}^+(q,\mathbb{Z})$, where $q$ is the quadratic form
\[ q(x_1,x_2,x_3,x_4) =  x_1^2+ x_2^2+x_3^2 - m\,x_4^2, \]
for a prime $m$ congruent to $-1$ modulo $8$. Our result covers infinitely more commensurability classes of compact arithmetic hyperbolic $3$--manifolds, and the constant $D$ of Theorem \ref{Sec:ArEx:T1} emerges from dealing with the larger class of forms $q$ one encounters in obtaining this generalization.
For instance, the $5/1$ Dehn filling of the census manifold $m306$ is a closed, arithmetic hyperbolic manifold for which we can now give an explicit upper bound on the index of a special subgroup (see Example \ref{ex:m306}).

Our second main result, Theorem \ref{red sauce}, extends results of Patel \cite{Patel} that give explicit linear bounds on the \textit{geodesic residual finiteness growths} of certain hyperbolic $3$-- and $4$--manifold groups.  The study of this invariant and its relation to the existing literature on residual finiteness growth is well introduced in \cite[\S 1]{Patel}.  For now we will let it suffice to record that work of Bou-Rabee--Hagen--Patel \cite{BRHP} implies that for every closed hyperbolic $3$--manifold $M$, the geodesic residual finiteness growth of $\pi_1 M$ is \textit{at most linear}: there is a constant $K$ such that for each loxodromic element $\alpha$ of $\pi_1 M$ there exists $H<\pi_1 M$ with $\alpha\notin H$ and\begin{align}\label{haha}
	[\pi_1 M : H] \leq K\,\ell(\alpha), \end{align}
where $\ell(\alpha)$ is the translation length of $\alpha$ (i.e. the length of the geodesic in $M$ representing $\alpha$). For a more detailed account of this implication, see Section \ref{geo vs rf}.

We are interested in obtaining explicit values for $K$, for given manifolds $M$. For instance when $M$ is a closed manifold that admits a totally geodesic immersion to a compact right-angled reflection orbifold, the main result of \cite{Patel} gives an explicit such value. Theorem \ref{red sauce} still requires $M$ to be closed, but allows it to immerse in a non-compact right-angled reflection orbifold of finite volume, such as $P_6$ above. Note that attempting to obtain explicit constants via the results of \cite{BRHP} presents several difficulties, including the non-effectiveness of the virtual special machinery.

Our bound depends on a choice of embedded horoballs.

\begin{dfn}\label{embedded}  For a polyhedron $P$ and an ideal vertex $v$ of $P$, we will say a horoball centered at $v$ is \textbf{embedded} in $P$ if it does not intersect the interior of any side of $P$ that is not incident on $v$.\end{dfn}

Here and below, the term side of a polyhedron $P$ refers specifically to a codimension-one face of $P$, following Ratcliffe (see \cite{Ratcliffe}, p.~198 and Theorem 6.3.1).

\newcommand\RedSauceThm{For $n\geq 2$, let $P$ be a right-angled polyhedron in $\mathbb{H}^{n+1}$ with finite volume and at least one ideal vertex, let $\Gamma_P$ be the group generated by reflections in the sides of $P$, and let $\calb$ be a collection of horoballs, one for each ideal vertex of $P$, that are each embedded in the sense of Definition \ref{embedded} and pairwise non-overlapping.  For a closed hyperbolic $m$--manifold $M$, $m\leq n$, that admits a totally geodesic immersion to $\mathbb{H}^{n+1}/\Gamma_P$, and any $\alpha\in \pi_1 M - \{ \mathrm{Id}_{\pi_1 M}\}$, there exists a subgroup $H'$ of $\pi_1 M$ such that $\alpha\notin H'$, and the index of $H'$ is bounded above by
$$  \frac{2v_n(1)}{V_{R+h_{\max}}} \sinh^{n}\left(R +d_{R+h_{\max}}\right)  \ell(\alpha), $$
where $v_n(1)$ is the volume of the $n$--dimensional Euclidean unit ball and:\begin{itemize}
	\item $\ell(\alpha)$ is the length of the unique geodesic representative of $\alpha$; 
	\item $R=\ln(\sqrt{n+1}+\sqrt{n})$;
	\item $h_{\max} = \ln(\cosh r_{\max})$, where $r_{\max}$ is the radius of the largest embedded ball in $M$; and 
	\item $d_{R+h_{\max}}$ and $V_{R+h_{\max}}$ are the diameter and volume, respectively, of the $(R+h_{\max})$--neighborhood in $P$ of $\overline{P-\bigcup\{B\in\calb\}}$.\end{itemize}}
\theoremstyle{plain}
\newtheorem*{red_sauce_thm}{Theorem \ref{red sauce}}
\begin{red_sauce_thm}\RedSauceThm\end{red_sauce_thm}

The bound above is the natural extension of \cite[Thm 3.3, Thm 4.3]{Patel}, with the role of the polyhedron $P$ there played here by a ``compact core'': the $(R+h_{\max})$--neighborhood of $\overline{P-\bigcup\{B\in\calb\}}$.  Because $h_{\max}$ appears here, the resulting bound depends not only on $P$ but also on $M$, unlike in \cite{Patel}.  This reflects the fact that we use the radius of the largest embedded ball in $M$ to control its interaction with the thin part of $P$, where the techniques of \cite{Patel} break down.

Theorem \ref{red sauce} applies to a significantly larger class of examples than \cite{Patel}.  In particular, compact right-angled polyhedra exist in $\mathbb{H}^n$ only for $n\leq 4$, whereas Theorem \ref{red sauce} covers the 6--dimensional finite-volume example $P_6$ of Theorem \ref{Sec:ArEx:T1} and other examples up to dimension at least eight (see eg.~\cite{PotVin}).  As Theorem \ref{Sec:ArEx:T1} shows, having more dimensions to work with allows one to produce totally geodesic immersions of more hyperbolic $3$--manifolds.

Given its prominent role in Theorem \ref{Sec:ArEx:T1}, and hence in the application of Theorem \ref{red sauce} to actual examples, we find it useful to look a bit more closely at the 6--dimensional right-angled polyhedron $P_6$ mentioned there.  In Section \ref{constants}, we collect enough geometric data on $P_6$ and the simplex $\sigma$ that generates it to give explicit formulas bounding the constants $d_{R+h_{\max}}$ and $V_{R+h_{\max}}$ appearing in Theorem \ref{red sauce} when $P = P_6$.  Combining this with Theorems \ref{Sec:ArEx:T1} and \ref{red sauce} yields:

\newcommand\EffectiveCor{Let $M = \mathbb{H}^3/\Gamma$ be a closed arithmetic hyperbolic $3$--manifold such that $\Gamma$ is commensurable with $\mathrm{SO}^+(q,\mathbb{Z})$ for some $\mathbb{Q}$--defined form $q$.  For any $\epsilon>0$ and any $\alpha\in\Gamma-\{ \mathrm{Id}_\Gamma\}$, there exists a subgroup $H'$ of $\Gamma$ such that $\alpha\notin H'$, and the index of $H'$ is bounded above by
\[  2^73^45 \cdot C_\epsilon\cdot D\cdot \mathrm{vol}(M)^\epsilon \cdot \frac{v_5(1)}{V_0}\sinh^5\left(2(2R+d_{\max}+\ln p^{-1}(\mathrm{vol}(M)))\right) \ell(\alpha), \]
where $v_5(1) = 8\pi^2/15$ and:\begin{itemize}
\item $\ell(\alpha)$ is the length of the unique geodesic representative of $\alpha$;
\item $R = \ln(\sqrt{6}+\sqrt{7})$;
\item $C_\epsilon\le 2^{\epsilon C'_\epsilon+2}11^2d_k^{A_1\omega(d_k)+3/2}$ with notation as in Equation \eqref{eqn:epsprime}, Proposition \ref{prop:indexbound}, and Proposition \ref{prop:indexbound2};
\item $D\le Ad^{2.975\cdot 10^{13}}$, where $A$ is an absolute, effectively computable constant and $d=z_1z_2z_3z_4$ (see Proposition \ref{prop:conjbound} for notation);
\item $V_0 = \frac{2^{2.5}\pi^3-3^4}{2^{2.5}\cdot5\cdot3}\approx 1.112$ and $d_{\max} = \cosh^{-1}(\sqrt{3})$, see Corollary \ref{able}; and
\item $p(x) = \frac{1}{5}x^5 - \frac{2}{3}x^3 + x - \frac{8}{15}$.\end{itemize}}
\newtheorem*{effective_cor}{Corollary \ref{effective}}
\begin{effective_cor}\EffectiveCor\end{effective_cor}

This formula yields explicit numerical bounds for the geodesic residual finiteness growths of actual examples.  For instance, in Example \ref{m306 taketwo} we build on Example \ref{ex:m306} to give an explicit value for the constant $K$ appearing in Equation \eqref{haha} when $M$ is the $5/1$ Dehn filling of the census manifold m306.  The value given is approximately $7\cdot 10^{150}$, which may well be non-optimal.  But we emphasize that it was produced by an effective process that can produce such a number for any closed manifold $M=\mathbb{H}^3/\Gamma$, where $\Gamma$ is an arithmetic lattice commensurable with $\mathrm{SO}^+(q,\mathbb{Z})$ for a bilinear form $q$ of signature $(3,1)$ with coefficients in $\mathbb{Q}$.

Results of this form depart from the existing literature on \textit{residual finiteness growth} (not modified by ``\textit{geodesic}'') in their degree of precision.  The more general notion introduced by Bou-Rabee, \cite{Bou}, measures the efficiency with which non-identity elements of an arbitrary residually finite group can be excluded from finite-index subgroups in terms of their word lengths. The word length of an element $g$ of a finitely generated group $\Gamma$ must be computed with respect to a finite generating set $X$ for $\Gamma$, and it depends on the choice of this set up to additive/multiplicative constants. The literature on residual finiteness growth thus employs a notion of asymptotic growth that is invariant under change of generating set, and as such all linear functions have the same growth. (We expand on this in Section \ref{geo vs rf}; for more detail see eg.~\cite{Bou} or \cite{BRHP}.)

When $\Gamma=\pi_1 M$ for a finite-volume hyperbolic manifold $M$ of dimension at least three, the geodesic length function $\ell\co \Gamma\to [0,\infty)$ offers a measure of complexity of elements that is an invariant of $\Gamma$, by Mostow rigidity. And when $M$ is closed, it follows from the \v{S}varc--Milnor lemma that the residual finiteness and geodesic residual finiteness functions have the same asymptotic growth rate \cite[Lem 6.1]{Patel}. The geodesic residual finiteness function thus gives a canonical choice for measuring the residual finiteness growth in these cases. Since this growth is at most linear for all closed $M$ \cite{BRHP} (cf.~Section \ref{geo vs rf}), it is natural to seek finer information of the form described in Sections \ref{pplus} and \ref{constants}.

It is not known whether the (geodesic) residual finiteness growth of closed hyperbolic manifolds is \textit{at least} linear. Another unresolved question around our work arises from considering non-compact but finite-volume hyperbolic $n$--manifolds.  For such $M$, the results of \cite{BRHP} still imply that the residual finiteness growth of $\pi_1 M$ is at most linear, but there are currently no upper or lower bounds on the geodesic residual finiteness growth of any such manifold in the literature. For a more detailed discussion suggesting that the residual finiteness and geodesic residual finiteness functions need not have the same growth rate, see \cite[\S 6.2]{Patel}. 

\subsection*{Acknowledgements:}
The authors thank David Ben McReynolds and Jeffrey S. Meyer for their contributions to the genesis of this project. We are greatly indebted to Ben, without whom this paper would not have been possible. His guidance and perspective have enriched many aspects of it. We also thank Benjamin Linowitz for correspondence leading to the effective part of the proof of Proposition \ref{prop:indexbound2}.
Finally, the authors thank the anonymous referee whose suggestions improved the exposition of this paper.


\section{Arithmetic Background}

This section introduces the requisite notation and terminology for the proof of Theorem \ref{Sec:ArEx:T1}. The reader in need of a more detailed treatment of any portion of this material is referred to the book of Maclachlan--Reid \cite{MR} where it is all thoroughly covered. 


\subsection{Preliminary Notation}
Throughout the rest of the text, the field $k$ will always be either the rational numbers $\Q$, an imaginary quadratic extension $\Q(\sqrt{-d})$ for $d\in\mathbb{N}$ a square-free number, or the real numbers $\R$.
When $k=\Q(\sqrt{-d})$, we denote its ring of integers by $\mathcal{O}_d$.
Given a rational prime $p$, $p\mathcal{O}_d$ factors into prime ideals in three possible ways -- as a single prime $\frak{p}$, as a product $\frak{p}\overline{\frak{p}}$, or as $\frak{p}^2$ for a prime $\frak{p}$.
In all cases we say that $\frak{p}$ (or $\overline{\frak{p}}$) lies over $p$.
Moreover, in the first case we say that $p$ is inert, in the second case we say that $p$ splits, and in the last case we say that $p$ is ramified.
When $k=\Q(\sqrt{-d})$, we define the norm of a prime ideal $\frak{p}$ in $\mathcal{O}_d$ as $\Nr(\frak{p})=|\mathcal{O}_d/\frak{p}|$.
Notice that for $\frak{p}$ lying over a rational prime $p$, $\Nr(\frak{p})=p^2$ when $p$ is inert and $\Nr(\frak{p})=p$ if $p$ splits or is ramified.


\subsection{Some Lattices in $\SL(2,\C)$ from Quaternion Algebras}\label{subsec:quat}
When discussing quaternion algebras, we will always assume that $k=\Q(\sqrt{-d})$ for $d\in\mathbb{N}$ a square-free number.
A \textbf{quaternion algebra} over $k$ is a $4$-dimensional algebra $k[1,I,J,IJ]$ with multiplication determined by the rules
$$I^2=\alpha,\quad J^2=\beta,\quad IJ=-JI,$$
for some $\alpha,\beta\in k^*$.
We typically refer to quaternion algebras using the compact notation
$$\mathcal{A}=\left(\frac{\alpha,\beta}{k}\right),$$
called a \textbf{Hilbert symbol}.
We remark that the Hilbert symbol is not unique, that is, the same quaternionin algebra can be represented by several different Hilbert symbols.

Given $g=w+xI+yJ+zIJ\in\mathcal{A}$, define the norm of $g$ by $\Nr_\mathcal{A}(g)=w^2-\alpha x^2-\beta y^2 +\alpha\beta z^2$.
If there is no non-trivial $g\in\mathcal{A}$ with $\Nr_\mathcal{A}(g)=0$ then we call $\mathcal{A}$ a \textbf{division algebra}, otherwise $\mathcal{A}\cong\Mat(2,k)$ and we call $\mathcal{A}$ a \textbf{matrix algebra}.
Given any prime $\frak{p}$ in $\mathcal{O}_d$ we may form the local field $k_\frak{p}$ which is a finite extension of the p-adic field $\Q_p$, where $\frak{p}$ lies over $p$.
Taking the tensor product $\mathcal{A}_\frak{p}=\mathcal{A}\otimes_k k_\frak{p}$ yields either the matrix algebra $\Mat(2,k_\frak{p})$ or a unique isomorphism class of division algebras over $k_\frak{p}$.
In the former case we say that $\mathcal{A}_\frak{p}$ \textbf{splits} and in the latter case we say that $\mathcal{A}_\frak{p}$ is \textbf{ramified}.
We use the notation $\Ram_f(\mathcal{A})$ to denote the collection of primes in $\mathcal{O}_d$ such that $\mathcal{A}_\frak{p}$ is ramified and $r_f$ to denote its cardinality; $r_f$ is always a finite number.
By the Albert--Brauer--Hasse--Noether theorem \cite[Thm 2.7.5]{MR}, the isomorphism class of $\mathcal{A}$ is uniquely determined by the set $\Ram_f(\mathcal{A})$. Here we recall that we are assuming $k=\Q(\sqrt{-d})$, so that the infinite places do not play a role.

To build lattices in $\SL(2,\C)$, note that there is an embedding of $\mathcal{A}$ into $\Mat(2,\C)$ given by
$$\varphi:w+xI+yJ+zIJ\mapsto\begin{pmatrix}
w+x\sqrt{\alpha}&\beta(y+z\sqrt{\alpha})\\
y-z\sqrt{\alpha}&w-x\sqrt{\alpha}
\end{pmatrix},$$
from which it is clear that $\Nr_\mathcal{A}(g)=\det(\varphi(g))$.
Consequently we have an embedding of the norm one elements $\mathcal{A}^1\to\SL(2,\C)$.
A subring $\mathcal{O}<\mathcal{A}$ is called an \textbf{order} if it is also an $\mathcal{O}_d$-lattice and $\mathcal{O}\otimes_{\mathcal{O}_d}k\cong\mathcal{A}$ \cite[\S 2.2]{MR}.
We use $\mathcal{O}^1$ to denote the norm one elements of $\mathcal{O}$.
Under the embedding above, $\varphi(\mathcal{O}^1)$ is an arithmetic lattice in $\SL(2,\C)$ \cite[\S 11]{BHC}.
Moreover, $\varphi(\mathcal{O}^1)$ is a cocompact lattice if and only if $\mathcal{A}$ is a division algebra, which in our setting is precisely the condition that $r_f>0$.
We also call any lattice commensurable with $\varphi(\mathcal{O}^1)$ arithmetic.
In the sequel, we will frequently suppress the embedding $\varphi$ and assume that $\mathcal{A}$ comes equipped with a fixed embedding into $\Mat(2,\C)$.

It is worth mentioning that by letting $d$ and $\Ram_f(\mathcal{A})$ vary, the above construction produces infinitely many commensurability classes of lattices.
Moreover, in the non-cocompact setting this construction produces all commensurability classes of arithmetic lattices.
That is to say, that all non-cocompact arithmetic lattices are commensurable with $\SL(2,\mathcal{O}_d)$ as $d$ varies over all square-free natural numbers \cite[Thm 8.2.3]{MR}.
In the compact setting, this is not the case as one needs to allow for fields other than just imaginary quadratic extensions to construct all arithmetic lattices in $\SL(2,\C)$. We refer the interested reader to \cite{MR} for a more detailed discussion. 


\subsection{Lattices in $\SO(3,1;\R)$ of Simplest Type}\label{subsec:quad}
While discussing quadratic forms we will always make the simplifying assumption that $k=\Q$ or $k=\R$.
By a quadratic form over $k$ of dimension $n$, we mean a homogeneous polynomial of degree $2$ in $n$ variables with coefficients in $k$.
We say that two quadratic forms $q$ and $q'$ of dimension $n$ are \textbf{$k$-isometric} if there exists $P\in\GL(n,k)$ such that $q(x)=q'(Px)$ for all $x\in k^n$.
This is equivalent to requiring that $P^TA_{q'}P=A_{q}$ for some $P\in\GL(n,k)$, where $A_q$, $A_{q'}$ are matrix representatives of $q$, $q'$.
Any quadratic form $q$ over $k$ of dimension $n$ is $k$-isometric to a diagonal form $q'(x)=a_1x_1^2+\dots+a_nx_n^2$, where $a_i\in k$ for $i = 1, \dots, n$, see eg.~\cite[Lem 0.9.4]{MR}.
We frequently use the notation $q=\langle a_1,\dots,a_n\rangle$ to describe a choice of diagonalization of $q$ and when $a_i\neq 0$ for all $i=1,\dots, n$, we say that $q$ is \textbf{non-degenerate}.
The notion of non-degeneracy is invariant under choice of diagonalization.

For a non-degenerate quadratic form $q$ of dimension $n$ over $\Q$, the associated special orthogonal group is given by
$$\SO(q,k)=\{A\in\SL(n,k)\mid q(Ax)=q(x)\text{ for all }x\in k^n\},$$
where we still make the assumption that $k$ is either the rational or real numbers.
When $k=\Q$, we may further restrict to integral matrices and define the group
$$\SO(q,\Z)=\SO(q,\Q)\cap\SL(n,\Z).$$
In the sequel, we use the notation $q_{n,1}$ to denote the specific quadratic form $q_{n,1}(x)=x_1^2+\dots+x_n^2-x_{n+1}^2$, which we sometimes referred to as the standard quadratic form.
Then $\SO(n,1;\R)$ is defined by
$$\SO(n,1;\R)=\SO(q_{n,1},\R)=\{A\in\SL(n+1,\R)\mid q_{n,1}(Ax)=q_{n,1}(x)\text{ for all }x\in\R^{n+1}\}.$$
Given another non-degenerate, diagonal quadratic form $q=\langle a_1,\dots,a_n\rangle$, Sylvester's law of inertia implies that, up to $\R$-isometry, $q$ is completely determined by its \textbf{signature} $(n^+_q,n^-_q)$.
Here $n^+_q$ (resp. $n^-_q$) is the number of $a_i$ that are positive (resp. negative).
In particular when $q$ has signature $(n^+_q,n^-_q)$, it is $\R$-isometric to the diagonal form
$$x_1^2+\dots +x^2_{n^+_q}-x^2_{n^+_q+1}-\dots-x^2_{n}.$$
Consequently if $q$ is a non-degenerate quadratic form over $\Q$ of signature $(3,1)$ then $\SO(q,\R)\cong \SO(3,1;\R)$ and, under this isomorphism, we have an embedding of $\SO(q,\Z)$ into $\SO(3,1;\R)$.
We call this subgroup and any subgroup commensurable with it in $\SO(3,1;\R)$ an arithmetic lattice.
For such a quadratic form, we say that $q$ is \textbf{isotropic} over $\Q$ if there exists a non-trivial $x\in\Q^4$ such that $q(x)=0$ and we call it \textbf{anisotropic} over $\Q$ otherwise.
The lattice $\SO(q,\Z)$ is cocompact if and only if $q$ is anisotropic over $\Q$.

It will be worthwhile for us to make a few comments about the commensurability classification of the lattices $\SO(q,\Z)$.
Let $q=\langle a_1,\dots,a_{n+1}\rangle$ be a non-degenerate quadratic form of signature $(n,1)$ over $\Q$ (so that $a_1,\dots,a_n\in\Q_{>0}$ and $a_{n+1}\in\Q_{<0}$).
Then we define the \textbf{discriminant} $\disc(q)$ of $q$ to be the product $a_1a_2\dots a_{n+1}$ considered as an equivalence class in $\Q^*/(\Q^*)^2$.
Fixing a rational prime $p$, the \textbf{Hilbert symbol} $(a_i,a_j)_p$ is defined by
$$(a_i,a_j)_p=\begin{cases}
1,&\text{ if $a_ix^2+a_jy^2=z^2$ has a non-trivial solution in $\Q_p^3$}\\
-1,&\text{ else}
\end{cases},$$
and the \textbf{Hasse--Witt invariant} $\epsilon_p(q)$ of the quadratic form $q$ is given by the product
$$\displaystyle\epsilon_p(q)=\prod_{i<j}(a_i,a_j)_p.$$
Given a non-degenerate quadratic form $q$, the Hasse--Witt invariant and discriminant are both invariants of the isometry class of $q$ and hence independent of choice of diagonalization.
The following two theorems give a complete commensurability classification of the lattices $\SO(q,\Z)$.
\newtheorem*{HMthm}{Hasse-Minkowski Theorem}
\begin{HMthm}[See eg.~\cite{Scharlau}, Chapter 6, Corollary 6.6]
Let $q$ and $q'$ be non-degenerate quadratic forms over $\Q$ of signature $(n,1)$, then $q$ and $q'$ are $\Q$-isometric if and only if $\epsilon_p(q)=\epsilon_p(q')$ for all rational primes $p$ and $\disc(q)=\disc(q')$ as classes in $\Q^*/(\Q^*)^2$.
\end{HMthm}
\begin{thm1}[See eg.~\cite{MThesis}, \S 4.3 and 4.4] 
Let $q$ and $q'$ be non-degenerate quadratic forms over $\Q$ of signature $(n,1)$, then $\SO(q,\Z)$ is commensurable with $\SO(q',\Z)$ if and only if $q$ and $q'$ are similar, i.e. there exists $\lambda\in\Q^*$ such that $q$ is $\Q$-isometric to $\lambda q'$.
\end{thm1}

As it will be useful for later, we conclude this subsection by listing a few properties of the Hilbert symbol (see for instance \cite[Ch III, \S 1]{Serre}):
\begin{enumerate}
\item $(x,y)_p=(y,x)_p$,
\item $(xx',y)_p=(x,y)_p(x',y)_p$,
\item $(x,y)_p=1$ for all but finitely many primes $p$,
\item $(x,y)_p=1$ if $p\neq 2$ and $p$ does not divide $x$, $y$,
\item $(x,x)_p=(x,-1)_p$,
\item $(x,y)_p=(x,-xy)_p$.
\end{enumerate}


\subsection{Lattices in $\PSL(2,\C)$ and $\SO^+(3,1;\R)$}
Quotienting $\SL(2,\C)$ by its center $\langle\pm\Id\rangle$ we obtain the group $\PSL(2,\C)$, which we call the \textbf{projective special linear group}.
Taking the similar quotient for $\SO(3,1;\R)$ we obtain the \textbf{projective special orthogonal group}, which we denote by $\SO^+(3,1;\R)$ for notational consistency with \cite{MR}.
The groups $\PSL(2,\C)$ and $\SO^+(3,1;\R)$ are isomorphic and also isomorphic to the orientation preserving isometries of hyperbolic $3$-space, $\mathrm{Isom}^+(\mathbb{H}^3)$.
Given a subgroup $\Delta$ in either $\SL(2,\C)$ or $\SO(3,1;\R)$, we denote by $P(\Delta)$ its image in the quotient and we call a lattice $P(\Gamma)$ \textbf{arithmetic} when $\Gamma$ is arithmetic in either $\SL(2,\C)$ or $\SO(3,1;\R)$.
We define the \textbf{covolume} of a lattice $\Gamma$ in either $\PSL(2,\C)$ or $\SO^+(3,1;\R)$ to be the volume of the quotient $\mathbb{H}^3/\Gamma$, which we write as $\cvol(\Gamma)$.

It is important to note that though $\PSL(2,\C)$ and $\SO^+(3,1;\R)$ are isomorphic, the commensurability classes of lattices arising from Subsections \ref{subsec:quat} and \ref{subsec:quad} are not in one to one correspondence.
Indeed, the commensurability classes of lattices arising from Subsection \ref{subsec:quad} are a proper subclass of those from Subsection \ref{subsec:quat}.
We briefly mention that this proper subclass can be described as the commensurability classes of lattices where there is a representative in the isomorphism class of the quaternion algebra with Hilbert symbol
$$\mathcal{A}=\left(\frac{\alpha,\beta}{\Q(\sqrt{-d})}\right),$$
where $\alpha,\beta$ are rational. Though we do not attempt to explain the details here, this is well known and can be shown, for instance, using the discussion in \cite[\S 10.2]{MR}.


\section{Arithmetic lattice bounds}\label{arithmetic bounds}

The entirety of this section is devoted to proving the following effective theorem:

\begin{thm}\label{Sec:ArEx:T1}\EmbeddingProp\end{thm}

We prove Theorem \ref{Sec:ArEx:T1} in Subsection \ref{thebigkahuna} as a straightforward consequence of Theorem \ref{prop:deltaconstruct} and Proposition \ref{prop:conjbound}, proved in \ref{sec:prasadvol} and \ref{sec:quadconj} (respectively). The proof strategy of each of these results is described in their respective subsections.


\subsection{}\label{sec:prasadvol}

The goal of this section is to use Borel's volume formula to estimate the index of a particular subgroup of $\Gamma$ as a function of $q$ and the volume of $\mathbb{H}^3/\Gamma$. Specifically, we prove the following.

\begin{thm}\label{prop:deltaconstruct}
Let $\Gamma$ be an arithmetic lattice commensurable with $\SO^+(q,\Z)$, where $q$ has signature $(3,1)$ over $\R$.
Then there is a $\Q$--defined quadratic form $q'$, a subgroup $\Delta$ of $\Gamma$, and an element $g\in\GL(4,\Q)$ such that $\SO^+(q',\Q)=g\SO^+(q,\Q)g^{-1}$, $g\Delta g^{-1}<\SO^+(q',\Z)$, and for any $\epsilon>0$, $[\Gamma:\Delta]\le C_\epsilon~\! V^{\epsilon},$ where $V$ is the volume of $\mathbb{H}^3/\Gamma$ and $C_\epsilon$ is a constant depending only on $\epsilon$ and $q$.
\end{thm}

\noindent To prove Theorem \ref{prop:deltaconstruct}, we will instead consider $\Gamma$ as a subgroup of $\PSL(2,\C)$.
To this end, we will first prove Propositions \ref{prop:indexbound}, \ref{prop:indexbound2}, and Corollary \ref{cor:indexbound3} giving a similar index bound for lattices arising from quaternion algebras. We then use an explicit isomorphism of $\SO(q,\Q)$ with $\mathcal{A}^*_q/k_q^*$ for a certain quaternion algebra $\mathcal{A}_q$ over a particular imaginary quadratic field $k_q$, which induces an isomorphism of $\PSL(2,\C)$ with $\SO^+(3,1;\R)$, to transfer the index bounds back to the orthogonal groups. For the bounds for lattices in quaternion algebras, we require the work of Borel \cite{borel} on volumes of lattices in $\PSL(2,\C)$.

Throughout, $k/\Q$ will be an imaginary quadratic extension and $\mathcal{A}/k$ will be a quaternion algebra with a fixed embedding $\mathcal{A} \to \mathrm{M}(2,\C)$ and with $P(-)$ denoting the projectivization under this embedding. Given an order $\mathcal{O} < \mathcal{A}$, we know that $\mathcal{O}^1 < \SL(2,\C)$ and $P(\mathcal{O}^1)<\PSL(2,\C)$ are arithmetic lattices. For any other order $\mathcal{O}_0$, the groups $\mathcal{O}^1,\mathcal{O}_0^1$ and $P(\mathcal{O}^1),P(\mathcal{O}_0^1)$ are commensurable. Finally, $\Lambda < \PSL(2,\C)$ will denote a lattice that is commensurable with $P(\mathcal{O}^1)$ for some order $\mathcal{O}<\mathcal{A}$.

\subsubsection{Maximal orders, Eichler orders, and maximal lattices}

Borel \cite{borel} proved that $\Lambda$ is contained in only finitely many maximal arithmetic lattices and that all maximal arithmetic lattices arise as normalizers of a specific class of Eichler orders, both of which we now describe.
An \textbf{Eichler order} is the intersection of two distinct maximal orders $\mathcal{O}_1$, $\mathcal{O}_2$ of $\mathcal{A}$, which we write as $\mathcal{E}=\mathcal{O}_1\cap\mathcal{O}_2$.
Recall that the \textbf{level} of an Eichler order $\mathcal{E}$ is the level of $\mathcal{E}_\frak{p}$ for each prime $\frak{p}$ of $\mathcal{O}_k$, that is to say that the level is the product $\prod_\frak{p}\frak{p}^{n_\frak{p}}$ where $n_\frak{p}$ is the distance between $(\mathcal{O}_1)_\frak{p}$ and $(\mathcal{O}_2)_\frak{p}$ in the tree $\mathcal{T}_\frak{p}$ associated to $\SL(2,k_\frak{p})$ (see \cite[\S 6.1, \S 6.6]{MR}). Given a fixed maximal order $\mathcal{O}$ and a finite (possibly empty) set of primes $S$ of $k$, which are disjoint from $\Ram_f(\mathcal{A})$, we may form the lattice $\Gamma_{S,\mathcal{O}}$ as follows. If $S=\emptyset$ then $\Gamma_{S,\mathcal{O}}=P(N(\mathcal{O}))$ where 
$$N(\mathcal{O})=\{x\in \mathcal{A}^*\mid x\mathcal{O}x^{-1}=\mathcal{O}\},$$
is the normalizer of $\mathcal{O}$ in $\mathcal{A}^*$ \cite[p~199]{MR}. If $S=\{\frak{p}_1,\dots,\frak{p}_r\}$, then first define $\mathcal{O}'_\frak{p}$ by
\[ \mathcal{O}'_\frak{p}=\begin{cases} \mathcal{R}_\frak{p},&\frak{p}=\frak{p}_i,\\ \mathcal{O}_\frak{p},&\frak{p}\notin S,\end{cases} \]
where $\mathcal{R}_\frak{p}$ is any choice of maximal local order which is distance one from $\mathcal{O}_\frak{p}$ in the tree $\mathcal{T}_\frak{p}$. Using the local-to-global principle \cite[V.2, Thm 2]{Weil}, the collection of local orders $\mathcal{O}'_\frak{p}$ define a global order $\mathcal{O}'$. Then the Eichler order $\mathcal{E}=\mathcal{O}\cap\mathcal{O}'$ of level $\prod_i\frak{p}_i$ allows us to define $\Gamma_{S,\mathcal{O}}=P(N(\mathcal{E}))$.
We understand the notation $\Gamma_{S,\mathcal{O}}$ to mean the lattice where some (arbitrary) choice of $\mathcal{R}_\frak{p}$ was made, however two such choices will always differ by $\mathcal{A}^1$ conjugacy and so the $\mathcal{A}^1$ conjugacy class of $\Gamma_{S,\mathcal{O}}$ only depends on $S$ and $\mathcal{O}$ (see \cite[\S 11.4]{MR} for more details). To denote a fixed choice of $\mathcal{E}$ we will sometimes say that $\Gamma_{S,\mathcal{O}}$ \emph{arises as the normalizer of $\mathcal{E}$}. It is also worth remarking that any Eichler order $\mathcal{E}$ associated to $\Gamma_{S,\mathcal{O}}$ is necessarily of square-free level.

Whereas not every $\Gamma_{S,\mathcal{O}}$ is a maximal arithmetic lattice, \cite{borel} has shown that all maximal arithmetic lattices arise as $\Gamma_{S,\mathcal{O}}$ for some finite set $S$ and some maximal order $\mathcal{O}$. Moreover, \cite{borel} and Chinburg--Friedman \cite{ChinFried} allow us to explicitly compute the volumes of lattices associated to Eichler orders of level $\prod_{i=1}^r\frak{p}_i^{n_i}$ as
\begin{equation}\label{eqn:eichvol}
\cvol(P(\mathcal{E}^1))=\frac{d^{3/2}_k\zeta_k(2)}{4\pi^2}\displaystyle\prod_{\frak{p}\in\Ram_f(\mathcal{A})}\left(\Nr(\frak{p})-1\right)\prod_{i=1}^r\Nr(\frak{p_i})^{n_i-1}\left(\Nr(\frak{p}_i)+1\right),
\end{equation}
and maximal arithmetic lattices as
\begin{equation}\label{eqn:maxvol}
\cvol(\Gamma_{S,\mathcal{O}})=\frac{d^{3/2}_k\zeta_k(2)}{8\pi^2 [k_\mathcal{A}:k]2^m}\displaystyle\prod_{\frak{p}\in\Ram_f(\mathcal{A})}\left(\frac{\Nr(\frak{p})-1}{2}\right)\prod_{\frak{p}\in S}\left(\Nr(\frak{p})+1\right),
\end{equation}
respectively.
In these equations, $m$ is an integer satisfying $0\le m\le |S|$, $d_k$ is the discriminant of $k$, $\zeta_k$ is the Dedekind zeta function of $k$, and $k_\mathcal{A}$ denotes the class field of $\mathcal{A}$ (this is $K(B)$ in the notation of \cite[\S 2]{ChinFried}).
The number theoretic quantites can be found in \cite[Ch 0,~11]{MR}, Equation \eqref{eqn:maxvol} appears, for instance, in \cite[Eqn 2.4]{LMM}, and Equation \eqref{eqn:eichvol} is a straightforward combination of \cite[Thm 11.1.3]{MR} and \cite[\S 11.2.2]{MR}.
Given $\epsilon > 0$, define
\begin{equation}\label{eqn:epsprime}
C'_\epsilon=14.5+2^{1/\epsilon+7},
\end{equation}
we then have the following proposition.

\begin{prop}\label{prop:indexbound}
Let $\Lambda$ be any arithmetic lattice contained in a maximal lattice $\Gamma_{S,\mathcal{O}}$ arising as the normalizer of an Eichler order $\mathcal{E}$ and let $V$ be the volume of $\mathbb{H}^3/\Lambda$.
Then for any $\epsilon>0$, $[\Lambda:\Lambda\cap P(\mathcal{E}^1)]\le C_{1,\epsilon}~\! V^{\epsilon}$, where $C_{1,\epsilon}>0$ is a constant depending solely on $\epsilon$ and the commensurability class of $\Lambda$.
Moreover, $C_{1,\epsilon}$ can be taken to be less than $2^{\epsilon C'_\epsilon+2}11^2d_k^{3/2}$, where $C'_\epsilon$ is as in Equation \eqref{eqn:epsprime}.
\end{prop}
\begin{proof}
By the above, $\Lambda < \Gamma_{S,\mathcal{O}}=P(N(\mathcal{E}))$ where $\mathcal{E}=\mathcal{O}\cap\mathcal{O}'$ has square-free level and $\mathcal{O}'$ is some fixed maximal order. In particular, it suffices to show that $[\Gamma_{S,\mathcal{O}}:P(\mathcal{E}^1)]\le C_{1,\epsilon}~\! V^{\epsilon}$. In this case, the level of $\mathcal{E}$ is simply the product of the primes in $S$. Combining \eqref{eqn:eichvol} and \eqref{eqn:maxvol} we see that
\begin{equation}\label{eqn:index}
[\Gamma_{S,\mathcal{O}}:P(\mathcal{E}^1)]=\frac{\cvol(P(\mathcal{E}^1))}{\cvol(\Gamma_{S,\mathcal{O}})}=2^{m+r_f+1}[k_\mathcal{A}:k]\le2^{|S|+r_f+1}[k_\mathcal{A}:k].
\end{equation}
We briefly note that $[k_\mathcal{A}:k]$ is bounded above by a constant only depending on $\mathcal{A}$. Specifically $[k_\mathcal{A}:k]\le h_k\le 242~\!d_k^{3/4}$, where $h_k$ is the class number of $k$ and depends only on $k$. Indeed, the first bound follows since $k_\mathcal{A}$ is contained in the narrow class field (whose degree is bounded above by the class number) and the second bound comes from work of Linowitz \cite[Lem 3.1]{Linowitz2}. As $d_k$ depends only on the commensurability class of $\Lambda$, we are reduced to showing that $r_f$ and $|S|$ behave logarithmically with respect to volume.

To accomplish this we use the techniques of \cite[Lem 2.5]{LMM} adapted to this setting.
By maximality and Equation \eqref{eqn:maxvol}, we have the trivial bound
\begin{equation}\label{eqn:covoleqn}
\frac{d^{3/2}_k\zeta_k(2)}{8\pi^2 [k_\mathcal{A}:k]2^m}\displaystyle\prod_{\frak{p}\in\Ram_f(\mathcal{A})}\left(\frac{\Nr(\frak{p})-1}{2}\right)\prod_{\frak{p}\in S}\left(\Nr(\frak{p})+1\right)\le \cvol(\Lambda)= V.
\end{equation}
As $0\le m\le |S|$, $\zeta_k(2)\ge 1$, and $[k_\mathcal{A}:k]\le 242~\!d_k^{3/4}$, we obtain that
\begin{equation}\label{eqn:normbnd}
\displaystyle\prod_{\frak{p}\in\Ram_f(\mathcal{A})}\left(\frac{\Nr(\frak{p})-1}{2}\right)\prod_{\frak{p}\in S}\left(\frac{\Nr(\frak{p})+1}{2}\right)\le \frac{2^4 11^2\pi^2}{d^{3/4}_k}V\le 2^4 11^2\pi^2 V.
\end{equation}
Notice that
$$\frac{1}{4^{r_f+|S|}}\displaystyle\prod_{\frak{p}\in\Ram_f(\mathcal{A})\cup S}\Nr(\frak{p})\le\displaystyle\prod_{\frak{p}\in\Ram_f(\mathcal{A})}\left(\frac{\Nr(\frak{p})-1}{2}\right)\prod_{\frak{p}\in S}\left(\frac{\Nr(\frak{p})+1}{2}\right),$$
and hence Equation \eqref{eqn:normbnd} yields
\begin{equation}\label{eqn:prodeqn}
\frac{1}{4^{r_f+|S|}}\displaystyle\prod_{\frak{p}\in\Ram_f(\mathcal{A})\cup S}\Nr(\frak{p})\le \alpha~\!V,
\end{equation}
where $\alpha=2^411^2\pi^2$.
Let $x=2^{2+1/\epsilon}$.
As $k$ is a quadratic extension, there can be at most $2\pi(x)$ primes $\frak{p}$ with $\Nr(\frak{p})\le x$, where $\pi(x)$ is the prime counting function.
Moreover, by \cite[Thm 4.6]{Apostol}
$$2\pi(x)\le 2 \frac{6x}{\ln(x)}=2 \frac{6x\log_2(e)}{\log_2(x)}\le 9\frac{2^{1/\epsilon+3}}{2+1/\epsilon},$$
and therefore Equation \eqref{eqn:prodeqn} implies
$$x^{r_f+|S|-9\frac{2^{1/\epsilon+3}}{2+1/\epsilon}}\le 4^{r_f+|S|}~\! \alpha ~\! V,$$
and consequently
\begin{align*}
r_f+|S|\le \frac{\log_2(\alpha)+ 9\frac{2^{1/\epsilon+3}}{2+1/\epsilon}\log_2(x) +\log_2(V)}{\log_2(x)-2}&\le \epsilon\left(\log_2(\alpha)+9\left(2^{1/\epsilon+3}\right)+\log_2(V)\right),\\
&\le \epsilon\left(14.5+2^{1/\epsilon+7}\right)+\epsilon\log_2(V).
\end{align*}
Revisiting Equation \eqref{eqn:index}, we therefore see that $[\Gamma_{S,\mathcal{O}}:P(\mathcal{E}^1)]=C_{1,\epsilon} V^{\epsilon}$ for some positive constant $C_{1,\epsilon}$ which depends only on $\epsilon$ and the field $k$.
Moreover, one can now see that $C_{1,\epsilon}\le 2^{\epsilon C'_\epsilon+2}11^2d_k^{3/2}$, with $C'_\epsilon$ as in Equation \eqref{eqn:epsprime}.
This completes the proof.
\end{proof}

We now define a preferred maximal order $\mathcal{O}^{\mathrm{std}}$ which we will use in the sequel. Let $\mathcal{O}^{\mathrm{std}}_\frak{p}$ denote the local maximal order defined by the unique maximal order at a ramified prime of $\mathcal{A}$ and by the maximal order $\mathrm{M}(2,\mathcal{O}_\frak{p})$ for all $\frak{p}\notin \Ram_f(\mathcal{A})$. Using the local-to-global principle, we define $\mathcal{O}^{\mathrm{std}}$ as the global maximal order of $\mathcal{A}$ with local completions $\mathcal{O}^{\mathrm{std}}_\frak{p}$.
Defining $\omega(n)$ to be the number of distinct prime divisors of a natural number $n$, we then have the following proposition, where the authors would like to thank Benjamin Linowitz for pointing out how to bound $C_2$ effectively.
Recall our terminology from the first paragraph of $\S 2.1.1$ of a maximal lattice $\Gamma=N(\mathcal{E})$ arising from a fixed Eichler order $\mathcal{E}$.

\begin{prop}\label{prop:indexbound2}
Let $\Gamma_{S,\mathcal{O}}$ be a maximal arithmetic lattice.
Then there exists an $h\in\mathcal{A}^*$, an Eichler order $\mathcal{E}_h$ such that the maximal lattice $h\Gamma_{S,\mathcal{O}}h^{-1}$ arises as the normalizer of $\mathcal{E}_h$, and an absolute constant $C_2$ depending only on the commensurability class of $\Gamma_{S,\mathcal{O}}$ such that $[P(\mathcal{E}_h^1):P((\mathcal{E}_h\cap\mathcal{O}_{\mathrm{std}})^1) ]\le C_2$. Moreover $C_2\le d_k^{\omega(d_k)A_1}$ where $A_1$ is an absolute, effectively computable constant.
\end{prop}
\begin{proof}
First we show that $C_2$ exists for maximal orders.
Recall that any quaternion algebra $\mathcal{A}$ has finite type number, i.e. there are only finitely many $\mathcal{A}^*$ conjugacy classes of maximal orders \cite[\S 6.7]{MR}.
Moreover given a conjugacy class of maximal orders $[\mathcal{O}]$ and any fixed maximal order $\mathcal{O}''$, there is a representative $\mathcal{O}'\in[\mathcal{O}]$ and a finite set of primes $S'$ (disjoint from $\Ram_f(\mathcal{A})$), such that $\mathcal{O}'_\frak{p}=\mathcal{O}''_\frak{p}$ for $\frak{p}\notin S'$ and $d(\mathcal{O}'_\frak{p},\mathcal{O}''_\frak{p})=1$ for $\frak{p}\in S'$ \cite[Cor 6.7.8]{MR}.
Let $\mathcal{O}''=\mathcal{O}^{\mathrm{std}}$, then using the above discussion and  \cite[\S 11.2.2]{MR} there is a fixed $S'$ such that
\begin{equation}\label{eqn:orderconj}
\displaystyle\min_{\mathcal{O}'\in[\mathcal{O}]}[\mathcal{O}':\mathcal{O}'\cap \mathcal{O}^{\mathrm{std}}]\le\prod_{\frak{p}\in S'}(\Nr(\frak{p})+1).
\end{equation}
We claim there is an absolute constant $C_2$ depending only on $\mathcal{A}$ such that 
$$\displaystyle\min_{\mathcal{O}'\in[\mathcal{O}]}[\mathcal{O}':\mathcal{O}'\cap \mathcal{O}^{\mathrm{std}}]\le C_2,$$
where $C_2$ can be chosen independent of the choice of conjugacy class $[\mathcal{O}]$.
Indeed, the bound for one conjugacy class is immediate from Equation \eqref{eqn:orderconj} and from this one simply lets $C_2$ be the maximum of the the righthand side as $[\mathcal{O}]$ varies over the finitely many conjugacy classes of maximal orders.

Now we show that $C_2$ exists for Eichler orders.
By construction of $\Gamma_{S,\mathcal{O}}$, it follows that $\Gamma_{S,h\mathcal{O}h^{-1}}=h\Gamma_{S,\mathcal{O}}h^{-1}$.
Let $h$ be such that $\mathcal{O}'=h\mathcal{O} h^{-1}$ minimizes the left side of Equation \eqref{eqn:orderconj} for the conjugacy class $[\mathcal{O}]$ and let $S'$ be the associated finite set of primes such that $d(\mathcal{O}'_\frak{p},\mathcal{O}^{\mathrm{std}}_\frak{p})=1$.
Then we may write any choice of Eichler order giving rise to $h\Gamma_{S,\mathcal{O}}h^{-1}$ as $\mathcal{E}_h=\mathcal{O}'\cap\mathcal{O}''$ where $\mathcal{O}''$ is defined locally by the rules
$$
\mathcal{O}''_\frak{p}=\begin{cases}
\mathcal{O}'_\frak{p},&\frak{p}\notin S,\\
\mathcal{R}_\frak{p},&\frak{p}\in S\backslash S',\\
\mathcal{O}^{\mathrm{std}}_\frak{p},&\frak{p}\in S\cap S',
\end{cases}$$
with $\mathcal{R}_\frak{p}$ being any choice of local maximal order associated to a vertex of distance $1$ in $\mathcal{T}_\frak{p}$.
Notice that $(\mathcal{E}_h)_\frak{p}$ is not contained in $\mathcal{O}_\frak{p}^{\mathrm{std}}$ if and only if $\frak{p}\in S'\backslash S$ and consequently
$$[\mathcal{E}_h:\mathcal{E}_h\cap \mathcal{O}^{\mathrm{std}}]=\prod_{\frak{p}\in S'\backslash S}(\Nr(\frak{p})+1),$$
As $S'$ is a fixed finite set (independent of $S$) we see that
$$[\mathcal{E}_h:\mathcal{E}_h\cap \mathcal{O}^{\mathrm{std}}]\le C_2.$$
Intersecting with the norm $1$ subgroup we get the similar bound
$$[\mathcal{E}^1_h:(\mathcal{E}_h\cap \mathcal{O}^{\mathrm{std}})^1]\le C_2.$$
From which the existential part of the proposition follows.

To bound $C_2$ effectively note that one can parametrize the conjugacy classes of maximal orders by instead parametrizing the $2$--torsion part of the idele class group \cite[Prop 4.1]{Linowitz}.
Interpreting the construction in \cite[\S 4]{Linowitz} properly, one can see that for any two conjugacy classes of maximal orders we have
$$[P(\mathcal{E}_h^1):P((\mathcal{E}_h\cap\mathcal{O}_{\mathrm{std}})^1) ]\le C_2\le \prod_{i=1}^r(\Nr(\frak{p}_i)+1),$$
for the $h$ constructed above and where the product is over the primes $\{\frak{p}_1,\dots,\frak{p}_r\}$ generating the $2$--torsion part of the idele class group.
It is a result of Gauss \cite{Gauss} that this is generated by $r=\omega(d_k)-1$ primes where $\omega(d_k)$ is the number of distinct prime divisors of $d_k$ (see also the discussion in Cohn \cite{Cohn}).
Moreover, it is a consequence of Artin reciprocity that the generators of the $2$--part of the idele class group are given by primes $\frak{p}_\sigma$ whose Artin symbol represents each conjugacy class in the Galois group $\Gal(k_\mathcal{A}/k)$, i.e. $\sigma=(k_\mathcal{A}/k,\frak{p}_\sigma)$.
By Lagarias--Montgomery--Odlyzko \cite[Thm 1.1]{LMO}, there is an absolute effectively computable constant $A_1$ such that $\Nr(\frak{p}_\sigma)\le d_k^{A_1}$ for all $\sigma\in\Gal(k_\mathcal{A}/k)$.
Consequently we see that
$$C_2= \prod_{i=1}^{\omega(d_k)-1}(\Nr(\frak{p})+1)\le d_k^{\omega(d_k)A_1},$$
giving the second claim and completing the proposition.
\end{proof}

\begin{cor}\label{cor:indexbound3}
Given an arithmetic lattice $\Lambda$ of covolume $V$, there is an element $h\in\mathcal{A}^*$ such that for any $\epsilon>0$ there is a constant $C_\epsilon$ which depends only on $\epsilon$ and the commensurability class of $\Lambda$ (i.e. only on $\mathcal{A}$ and $k$) such that $[\Lambda_h:\Lambda_h\cap P((\mathcal{O}^{\mathrm{std}})^1)]\le C_\epsilon V^{\epsilon}$ where $\Lambda_h=h\Lambda h^{-1}$.
Moreover, $C_\epsilon\le 2^{\epsilon C'_\epsilon+2}11^2d_k^{A_1\omega(d_k)+3/2}$.
\end{cor}
\begin{proof}
This is a simple combination of the preceding two propositions with $C_\epsilon=C_{1,\epsilon} C_2$.
Indeed, let $\Gamma_{S,\mathcal{O}}$ be any maximal arithmetic lattice containing $\Lambda$ and let $h$ be the element supplied by Proposition \ref{prop:indexbound2}.
Then
\begin{align*}
[\Lambda_h:\Lambda_h\cap P((\mathcal{O}^{\mathrm{std}})^1)]&\le [\Lambda_h:\Lambda_h\cap P((\mathcal{E}_h\cap \mathcal{O}^{\mathrm{std}})^1)]\\
&=[\Lambda_h:\Lambda_h\cap P(\mathcal{E}^1_h)][\Lambda_h\cap P(\mathcal{E}^1_h):\Lambda_h\cap P(\mathcal{E}^1_h\cap(\mathcal{O}^{\mathrm{std}})^1)]\le C_\epsilon V^{\epsilon},
\end{align*}
where the first index is bounded by an application of Proposition \ref{prop:indexbound} and the second is the bound from Proposition \ref{prop:indexbound2} transferred to the intersection with $\Lambda_h$.
\end{proof}

Before proceeding to the proof of Theorem \ref{prop:deltaconstruct} we recall a couple of useful facts.
Given the quadratic form $q$, if $q'$ is any $\Q$--isometric quadratic form then there is some $g\in\GL(4,\Q)$ such that the corresponding orthogonal groups are $\Q$--conjugate, i.e. $g\SO^+(q,\Q)g^{-1}=\SO^+(q',\Q)$.
Moreover if $q'$ is similar to $q$, that is to say $q'=\lambda q$ for some $\lambda\in\Q^*$, then in fact $\SO^+(q,\Q)=\SO^+(q',\Q)$.
These operations clearly preserve commensurability classes of lattices.
A discussion of this can be found, for instance, in \cite[Lem 4.2]{Meyer}.

We recall from \cite[\S 10.2]{MR} that given any quaternion algebra over an imaginary quadratic number field $\mathcal{A}=(\alpha,\beta)_{\Q(\sqrt{-d})}$, we have an exact sequence
$$\xymatrix{1\ar[r]&\Q(\sqrt{-d})^*\ar[r]&\mathcal{A}^*\ar[r]^-{\widetilde{\Phi}}&\ar[r]\SO(q',\Q)&1}$$
where $q'=\langle 1,d\alpha,d\beta,-d\alpha\beta\rangle$.
Tensoring this exact sequence by $\R$ and passing to index two subgroups induces an isomorphism $\Phi$ from $\PSL(2,\C)$ to $\SO^+(q,\R)\cong\SO^+(3,1;\R)$.
Moreover this isomorphism can be chosen to preserve our standard integral structure, i.e. such that $\Phi((\mathcal{O}^{\mathrm{std}})^1)\subset \SO(q,\Z)$.

\begin{proof}[Proof of Theorem \ref{prop:deltaconstruct}]
Given a quadratic form $q/\Q$ with signature $(3,1)$ over $\R$, we write $q=\langle z_1,z_2,z_3,-z_4\rangle$ where, up to similarity, we may assume that the $z_i$ are positive integers. Define $d=z_1z_2z_3z_4$, then $k=\Q(\sqrt{-d})$ is an imaginary quadratic field. Using the discussion of the preceding paragraph, we may transfer $\Gamma$ to an arithmetic lattice in $\PSL(2,\C)$, namely let 
$$g_1=\begin{pmatrix}
1/z_1&&&\\
&z_3z_4&&\\
&&z_2z_4&\\
&&&z_2z_3z_4
\end{pmatrix},$$
then $g_1$ produces a conjugate group $\SO^+(q',\Q)=g_1\SO^+(q,\Q)g_1^{-1}$ where $q'$ is the quadratic form given by
$$q'=\langle 1/z_1,z_2(z_3z_4)^2,z_3(z_2z_4)^2,-(z_2z_3z_4)^2z_4\rangle.$$
Notice that $q'$ is similar to the form 
\begin{equation}\label{eqn:newform}
q''=\langle 1,z_1z_2(z_3z_4)^2,z_1z_3(z_2z_4)^2,-z_1(z_2z_3z_4)^2z_4\rangle,
\end{equation}
which via $\widetilde\Phi$ is isomorphic to the quotient of the units of $\mathcal{A}=\left(\frac{z_3z_4,z_2z_4}{k}\right)$ modulo $k^*$.
Altogether this shows that there is some lattice $\Lambda<\PSL(2,\C)$ which is the image of $\Gamma$ under the composition of conjugation by $g_1$ and $\Phi$. By Corollary \ref{cor:indexbound3} there is $h\in\mathcal{A}^*$ such that for any $\epsilon>0$ there is a constant $C_\epsilon$ depending only on $\epsilon$ and the commensurability class of $q$ such that $[\Lambda_h:\Lambda_h \cap P((\mathcal{O}^{\mathrm{std}})^1)]\le C_\epsilon~\! V^{\epsilon}$, where $\Lambda_h=h\Lambda h^{-1}$.
Let $g_2=\Phi^{-1}(h)$, $g=g_2g_1$, and $\Delta=g^{-1}\Phi^{-1}((\Lambda_h\cap P((\mathcal{O}^{\mathrm{std}})^1)))g$.
Then $\Delta$ is a subgroup of $\SO^+(q,\Q)$ such that $[g\Gamma g^{-1}:g\Delta g^{-1}]\le C_\epsilon~\! V^{\epsilon}$, $g\SO^+(q,\Q) g^{-1}=\SO^+(q',\Q)$, and $g\Delta g^{-1}<\SO^+(q',\Z)$.
\end{proof}


\subsection{}\label{sec:quadconj}

In this subsection we give some preliminaries on the classical theory of quadratic forms, with the ultimate goal of proving the following proposition.

\begin{prop}\label{prop:conjbound}
Given $\SO^+(q,\Z)$ for $q=\langle z_1,z_2,z_3,-z_4\rangle$ a $\Q$-defined quadratic form of signature $(3,1)$ with each $z_i$ a positive integer, there exists a finite index subgroup $\Delta < \SO^+(q,\Z)$ and an injective homomorphism $\Delta \to \SO^+(6,1;\Z)$.
Moreover, the index of $\Delta$ in $\SO^+(6,1;\Z)$ is at most $D\le A~\! d^{{2.4\cdot 10^{15}}}$, where $d=z_1z_2z_3z_4$ and $A$ is an absolute, effectively computable constant.
If additionally any of the $z_i$ are equal to $1$ for $1\le i\le 3$, then we may take $D\le A~\!d^{4.25\cdot 10^{12}}$.
\end{prop}

\noindent To demonstrate the existence of such a finite index subgroup, it suffices to see that there exists a $\Q$--defined quadratic form $q_c$ of signature $(3,0)$  such that $q_c\oplus q$ is $\Q$--isometric to $q_{6,1}$.
We will point out how to construct such a form using local invariants below.
In order to then give an estimate for $D$ in Proposition \ref{prop:conjbound}, we will need to understand the $\Q$--isometry which takes $q_c\oplus q$ to $q_{6,1}$ and how it affects $\SO^+(q_c\oplus q,\Z)$.
More specifically, there exists $P\in\GL(7,\Q)$ such that the map $h\mapsto PhP^{-1}$ is an isomorphism of the groups $\SO^+(q_c \oplus q,\Q)$ and $\SO^+(6,1;\Q)$. The denominators of $P$ control the index of the subgroup of $\SO^+(q_c\oplus q,\Z)$ that has image in $\SO^+(6,1;\Z)$. Using largely elementary methods, we will provide explicit bounds for these denominators. 

Recall that quadratic forms up to isometry are completely determined by their rank, signature, discriminant, and Hasse--Witt invariants.
The latter two of these are elements of $\Q^*/(\Q^*)^2$ and $\{-1,1\}$ respectively (see the discussion in Subsection \ref{subsec:quad}).
By \cite{MThesis} (see also \cite[\S 7]{Meyer}), there exists a definite quadratic form $q_c$ of signature $(3,0)$ such that $q_c  \oplus q$ is $\Q$-isometric to $q_{6,1}$.
The invariants of $q_c$ are controlled by those of $q$ and $q_{6,1}$ and therefore $q_c$ is determined up to $\Q$--isometry by the following two conditions
\begin{align}
\label{eqn:disc} \disc (q_c) &= -\disc (q) \in \Q^*/(\Q^*)^2, \\ 
\label{eqn:hasse} 1 &= \epsilon_p(q)\epsilon_p(q_c).
\end{align}
To see Equation \eqref{eqn:hasse} note that the condition $\epsilon_p(q_{6,1})=\epsilon_p(q_c\oplus q)$ forces that
\begin{align}
\prod_{1\le i<j<7}(1,1)_p\prod_{i=1}^6(1,-1)_p&=\epsilon_p(q)\epsilon_p(q_c)\epsilon_p(\langle\disc(q),\disc(q_c)\rangle),\nonumber\\
1&=\epsilon_p(q)\epsilon_p(q_c)\label{eqn:CompFormHilb},
\end{align}
where on the righthand side we have used that
$$\epsilon_p(\langle\disc(q),\disc(q_c)\rangle)=(\disc(q_c),\disc(q))_p=(-\disc(q),\disc(q))_p=1,$$
with the last equality holding by definition of the Hilbert symbol.

Throughout the remainder of the section we use $-d$ to denote the product of the coefficients of $q$, that is to say that if $q=\langle z_1,z_2,z_3,-z_4\rangle$ for positive integers $z_i$ then $d=z_1z_2z_3z_4$.
Writing $q_c=\langle \alpha,\beta,\gamma\rangle$ for $\alpha,\beta,\gamma$ and square-free positive integers, we first prove the following effective lemma.

\begin{lemma}\label{lemma:CompFormLemma}
The form $q_c$ can be chosen so that $\alpha\beta\gamma$ is less than $D_0~\! d^{16}$, where $D_0$ is an absolute, effectively computable constant.
\end{lemma}

\begin{proof}
We will effectivize the proof given in Serre \cite{Serre} of the existence of an explicit global form $q_c$ constructed from the local invariants. In \cite[Prop IV.7]{Serre}, the complimentary form $q_c$ is proven to exist and constructed so that
\begin{equation}\label{eqn:serreform}
q_c=\langle x,c,cdx\rangle,
\end{equation}
with $d$ as previously defined and for well chosen $x$ and $c$ (for the reader's convenience we try to adopt as much of the notation in Serre's proof as possible).
To show how this will give a genuine complementary form, we briefly comment on the properties of $x$ and $c$ that we require.
First note that trivially
\[ \disc(q_c)=d=-\disc(q), \]
so Equation \eqref{eqn:disc} is satisfied.
In choosing $c$, we will require that if $p$ is such that $(-d,-1)_p=-\epsilon_p(q)$ then $[c]_p$ is in a different square class as $[-d]_p$ in $\Q^*_p/\Q_p^{*2}$.
We then choose $x$ such that
$$(x,-cd)_p=(c,-d)_p\epsilon_p(q),$$
or equivalently that
$$(x,-cd)_p(c,-d)_p=\epsilon_p(q).$$
Using some basic Hilbert symbol arithmetic (see the properties at the end of Subsection \ref{subsec:quad}), one sees that
$$\epsilon_p(q_c)=(x,c)_p(x,cdx)_p(c,cdx)_p=(x,-cd)_p(c,cdx^2)_p=(x,-cd)_p(c,-d)_p=\epsilon_p(q),$$
which is precisely the condition from Equation \eqref{eqn:hasse}.
That such integers $c$ and $x$ exist is proved in \cite[Thm III.4]{Serre}, the proof of which we now effectivize.
Alternatively, that $c$ and $x$ exist will follow from what is written below.
\normalsize

\textbf{Claim 1:} \emph{$c$ can be chosen so that $c$ divides $2d$.}

Let $A$ be the finite set of primes such that $(-d,-1)_p=-\epsilon_p(q)$ ($A$ is $S$ in Serre's notation), where we remark that the discriminant of $q_c$ is in the same square class as $d$.
Notice by properties of the Hilbert symbol, that the set $A$ must be contained in the set $\mathcal{P}$ of prime divisors of $2d$.
In particular, if $p\nmid 2d$ then $\epsilon_p(q)$ is by definition trivial (since it is a product of trivial Hilbert symbols) and similarly the righthand side of
$$(-d,-1)_p=(-1,-1)_p(d,-1)_p,$$
is a product of two trivial Hilbert symbols when $p\nmid 2d$.
For any $p\in\mathcal{P}$, define $z_p$ to be $0$ if the power of $p$ that divides $d$ is odd and $1$ if it is even.
Then $c=\prod_{p\in\mathcal{P}}p^{z_p}$ is not in the same square class as $-d$ for all $p$ in $\mathcal{P}$ and hence $A$ (see for instance \cite[p~18]{Serre}).
This completes the claim.

Given this $c$, we next construct $x$ as in Equation \eqref{eqn:serreform}.
Let $\epsilon'_p=(c,-d)_p\epsilon_p(q)$ then we find $x$ such that $(x,-cd)_p=\epsilon'_p$ for all primes $p$.
The existence of $x$ is given by \cite[III.Thm 4]{Serre}, from which we claim the following effective bound.

\textbf{Claim 2:} \emph{$x$ can be chosen so that $x\le D''_0~\!d^{6.5}$ for an absolute, effectively computable constant $D''_0$ depending only on $d$.}

Again following the notation of Serre, let $S$ be the set of primes that divide $2cd$ and let $T$ be the set of primes $p$ such that $\epsilon'_p=-1$.
By construction $S=\mathcal{P}$.

We first reduce in an effective manner from the general case in \cite[III.Thm 4]{Serre} to the special case that $S\cap T=\emptyset$.
Notice that for any prime $p$ dividing $-cd$ which does not divide $x_p$, if an even power of $p$ divides $cd$ then we have $(x_p,-cd)_p=(x_p,-1)_p=1$ and if an odd power of $p$ divides $cd$ then $(x_p,-cd)_p=(x_p,-p)_p$.
Therefore for any prime $p\in S$, we may choose $x_p$ as follows:
\begin{enumerate}
\item If $\epsilon'_p=1$, let $x_p=1$.
\item If $p=2$ and $\epsilon'_p=-1$, let $x_p=3$.
\item If $p$ is odd and $\epsilon'_p=-1$, let $x_p$ be the smallest prime quadratic non-residue modulo $p$.
\end{enumerate}
It is then clear that for such choices $(x_p,-cd)_p=\epsilon'_p$.
By standard approximation theorems, we find $x'\in\Z$ such that its image $x'_p$ in $\Q_p$ is in the same square class as $x_p$ for all $p\in S$.
Indeed, such an $x'$ is furnished by the Chinese remainder theorem (where for the dyadic prime, we work modulo $8$).
Since $S=\mathcal{P}$, $x'$ can therefore be chosen to be less than $8d$.
Now define $\epsilon_p''=(x,-cd)_p~\!\epsilon'_p$ and let $T'$ be the set of all primes $p$ such that $\epsilon''_p=-1$.
Notice now that $T'\cap S=\emptyset$ and $\prod_{\ell\in T'}\ell\le 8d$ by construction.
Hence the quantities
$$a=\prod_{\ell\in T'}\ell,\quad m=8\prod_{\substack{\ell\in S\\ \ell\neq 2}}\ell,$$
are relatively prime and both less than $8d$.
By an effective version of Linnik's theorem \cite{HeathBrown}, there is an absolute, effectively-computable constant $D'_0=D'_0(d)$ such that the smallest prime $q$ in the residue class of $a$ modulo $m$ is at most $D'_0 (8d)^{5.5}$.
We claim that $x=x'q$ then gives the desired $x$. 
Indeed, that it satisfies the requisite Hilbert symbol properties follows from \cite[III.Thm 4]{Serre} and moreover $$x\le 8dD'_0(8d)^{5.5}=D''_0d^{6.5},$$
completing Claim 2.

Putting these claims together, we see that the product
$$\alpha\beta\gamma=x^2c^2d\le (D''_0d^{6.5})^2(2d)^2d\le D_0d^{16}$$ for some absolute effectively computable constant $D_0$.
\end{proof}

We now estimate $D$ from the statement of Proposition \ref{prop:conjbound} by giving estimates on the $\Q$--isometry which takes $q_c\oplus q$ to $q_{6,1}$.
Given a rational number $s/t$ with $(s,t)=1$ and $t>0$, we will use the notation $\denom(s/t)=t$ in what follows where we define $\denom(0)=1$.
For any fixed $n\ge2$ and any quadratic form $g=\langle a_1,\dots,-a_n\rangle$, with $a_i$ positive integers, define
$$E_n(g)= 2\max_{1\le i\le n}|a_i|\left(3\sum_{i=1}^n|a_i|+3\right)^{n/2}.$$
We then have the following proposition.

\begin{prop}\label{prop:CongMat}
Let $a_1,\dots,a_n$ be positive integers and $g=\langle a_1,\dots,a_{n-1},-a_n\rangle$ be an integral quadratic form of signature $(n-1,1)$ for $n>1$ such that $g$ is $\Q$--isometric to $g_{n-1,1}=\langle 1,\dots,1,-1\rangle$.
If $A_g$ is the diagonal matrix representing $g$, then there exists a rational matrix $P\in\GL(n,\Q)$ such that $P^TA_gP$ is a matrix representing the integral diagonal form $\langle 1,b_2,\dots,-b_n\rangle$ of signature $(n-1,1)$, where each $b_i$ is a positive integer.
Moreover, $\lcm_{i,j}\{\denom(P_{ij})\}\le E_n(g)$.
\end{prop}

\begin{proof}
We first describe an algorithmic procedure for constructing $P$, then show that the $P_{ij}$ satisfy the requisite bound.
Our algorithm is the following:

\textbf{Step 1:} \emph{Changing the (1,1)--coefficient to $1$.} As $g$ is isometric to $g_{n-1,1}$, there is some $x\in\Q^n$ such that $g(x)=1$.
Let $v_1=x$ written as a column vector and let $v_2,\dots,v_n$ denote a basis for $v_1^\perp$ in $\Q^n$, where by clearing denominators and appropriately scaling we assume that for each $2\le i\le n$ that $v_i$ is in $\Z^n$ and that $v_i/\lambda\notin\Z^n$ for any natural number $\lambda>1$ (i.e. the $\gcd$ of all of $v_i$'s entries is $1$).
Defining
$$P_1=\begin{pmatrix}
\vrule&\vrule&&\vrule\\
v_1&v_2&\dots&v_n\\
\vrule&\vrule&&\vrule
\end{pmatrix},$$
it is clear that $P_1^TA_gP_1$ is a symmetric matrix with a $1$ as the $(1,1)$--entry.

\textbf{Step 2:} \emph{Diagonalize the resulting quadratic form.}
Let $g'$ be the quadratic form representing $Z=P_1^TA_gP_1$, then we use the Jacobi method to diagonalize $g'$.
To this end, let $Z_{k,k}$ denote the $k\times k$ minor of $Z$ which lies in the upper left corner and let $w_k$ be the $k\times 1$ vector which solves the system
\begin{equation}\label{eqn:Cramer}
Z_{kk}w_k=
\begin{pmatrix}z_{1,1}&\dots&\dots&z_{1,k}\\
\vdots&\ddots&&\vdots\\
\vdots&&\ddots&\vdots\\
z_{k,1}&\dots&\dots&z_{k,k}\end{pmatrix}
\begin{pmatrix}w_{1,k}\\
\vdots\\
w_{k-1,k}\\
w_{k,k}\end{pmatrix}
=
\begin{pmatrix}
0\\
\vdots\\
0\\
1\end{pmatrix}.
\end{equation}
The existence of such a $w_k$ follows from Cramer's rule.
Completing each $w_k$ to an $n\times 1$ vector by setting its last $n-k$ entries equal to $0$, we obtain a rational $n\times n$ upper triangular matrix
$$P_2=\begin{pmatrix}
\vrule&\vrule&&\vrule\\
w_1&w_2&\dots&w_n\\
\vrule&\vrule&&\vrule
\end{pmatrix}.$$
Write $c_i=\lcm_{1\le j\le n}\{\denom(w_{j,i})\}$, i.e. $c_i$ is the $\lcm$ of the denominators of the non-zero numbers in each column.
Defining
$$P_3=\begin{pmatrix}
c_1&&\\
&\ddots&\\
&&c_n\end{pmatrix},$$
yields an integral upper triangular matrix $P_2P_3$ such that $(P_2P_3)^TZP_2P_3$ represents the integral diagonal form $\langle 1,b_2,\dots,-b_n\rangle$ of signature $(n-1,1)$. Therefore setting $P=P_1P_2P_3$ gives a $\Q$--isometry such that $P^TA_gP$ represents $\langle 1,b_2,\dots,-b_n\rangle$.

As $P_2P_3$ is an integral matrix, to bound the denominators $\denom(P_{ij})$ it suffices to bound the coefficients coming from Step $1$. By construction, the entries of each $v_2,\dots,v_n$ are integral so it suffices to find a bound for the denominators in $v_1$. Thus, we must find a bound on the denominators of a solution to $g(x)=1$ for $x\in\Q^n$. For this, we use the following theorem of Cassels \cite{Cassels}, interpreted appropriately.

\begin{thm}[Cassels]\label{thm:Cassel}
Let $f=\langle k_1,\dots,k_m\rangle$ be an isotropic integral diagonal quadratic form in $m\ge 2$ variables, then there exists a non-trivial $y\in\mathbb{Z}^m$ such that $g(y)=0$ and
\begin{equation}\label{eqn:CasselBnd}
\max_{1\le i\le m}|y_i|\le \left(3\sum_{i=1}^m|k_i|\right)^{(m-1)/2}.
\end{equation}
\end{thm}

\noindent
Now let $\widetilde{g}$ be the augmented $(n+1)$--variable quadratic form $g\oplus \langle-1\rangle$.
As $\widetilde{g}$ is isotropic, Theorem \ref{thm:Cassel} produces a non-trivial $y\in\mathbb{Z}^{n+1}$ such that $\widetilde{g}(y)=0$ and such that $y$ satisfies the bound in Equation \eqref{eqn:CasselBnd} with $m=n+1$.
We now show how to use $y$ to produce $x$ in two cases.

\textbf{Case 1: $y_{n+1}\neq 0$.}
Then simply let
$x=(y_1/y_{n+1},\dots,y_n/y_{n+1})\in\Q^n,$
and clearly $g(x)=1$ by construction.

\textbf{Case 2: $y_{n+1}=0$.}
Then fix an index $i$ such that $y_i\neq 0$ and let $e_i=(0,\dots,0,1,0,\dots,0)$ with the $1$ in the $i$th spot.
Defining $x=e_i+\alpha\cdot (y_1,\dots,y_n)$ for $\alpha=(1-a_i)/2a_iy_i$, a routine computation shows that $g(x)=1$.

From this we claim that $\lcm_{i,j}\{\denom(P_{ij})\}\le E_n(g)$. Indeed, in the first case we have the bound
$$\lcm_i\{\denom(x_i)\mid x_i\neq 0\}\le y_{n+1}\le E_n(g),$$
and in the second case we have the bound
$$\lcm_i\{\denom(x_i)\mid x_i\neq 0\}\le 2a_iy_i\le E_n(g).$$
This therefore completes the proof.
\end{proof}

\noindent For any fixed $n\ge2$ and any quadratic form $g=\langle a_1,\dots,-a_n\rangle$, define 
$$F_n(g)= E_n(g)^{2n}n^{n/2}\prod_{k=1}^{n-1}E_n(g)^{2k+2}k^{k/2}.$$

\begin{cor}\label{cor:FLemma}
For given $g$ and $n$, the matrix $P$ constructed in Proposition \ref{prop:CongMat} has determinant bounded above by $F_n(g)$.
\end{cor}
\begin{proof}
As $g,n$ are fixed throughout the corollary, we set $E=E_n(g)$. First, $\det(P_1)$ is bounded above by $E^{2n} n^{n/2}$. Indeed, one can check that in the construction of $v_1$ from Proposition \ref{prop:CongMat} the numerator of each entry of $v_1$ is bounded above by $E^2$. Clearing denominators, each entry of $v_i$ can clearly be chosen to be bounded above by $E^2$ as well for $2\le i\le n$.
Hadamard's inequality then implies that $\det(P_1)\le E^{2n} n^{n/2}$.

As $P_2P_3$ is an upper triangular integral matrix, to bound $\det(P_2P_3)$ it suffices to give bounds on its diagonal coefficients.
Let $Z_{k,k}$ be the $k\times k$ minor of $Z=P_1^TA_gP_1$, $Z^{(j)}_{k,k}$ denote $Z_{k,k}$ with the $j$th column replaced by the column vector on the righthand side of Equation \eqref{eqn:Cramer}, and let $\mathrm{num}(\det(Z_{k,k}))$ denote the absolute value of the numerator of $\det(Z_{k,k})$.
Then the construction of $P_2$ using Cramer's rule gives that, for $k\ge 2$, each column vector $w_i$ has
$$c_k=\lcm_{j}\{\denom(w_{j,k})\}=\lcm_{j}\{\denom\left(\frac{\det(Z^{(j)}_{k,k})}{\det(Z_{k,k})}\right)\}\le E\cdot \mathrm{num}(\det(Z_{k,k})),$$
where we have used that $P_1$ and hence $Z_{k,k}$ and $Z^{(j)}_{k,k}$ have uniform denominator at most $E$.
Moreover, by construction, the $(k,k)$th diagonal coefficients of $P_2$ are given by $\det(Z_{k-1,k-1})/\det(Z_{k,k})$ where we use the convention that $Z_{0,0}=Z_{1,1}=1$.
Consequently
\begin{align*}
\det(P)&=\det(P_1)\det(P_2P_3) =\det(P_1)\prod_{k=1}^n\frac{\det(Z_{k-1,k-1})}{\det(Z_{k,k})}c_k,\\
&\le\det(P_1)\prod_{k=2}^n\frac{\det(Z_{k-1,k-1})\denom(\det(Z_{k,k}))}{\mathrm{num}(\det(Z_{k,k}))} \left(E\cdot \mathrm{num}(\det(Z_{k,k}))\right),\\
&\le E^{2n}n^{n/2}\prod_{k=2}^nE^2 \det( Z_{k-1,k-1}) \le E^{2n}n^{n/2}\prod_{k=1}^{n-1}E^{2k+2}k^{k/2} =F_n(g),
\end{align*}
where the last line is another application of Hadamard's inequality.
This completes the proof.
\end{proof}
\begin{cor}\label{cor:GLemma}
There is an explicit constant $G_n(g)$ depending only on $g$ and $n$, such that the coefficients $a_i$ and $b_i$ from Proposition \ref{prop:CongMat} differ by a factor of at most $G_n(g)$ for $2\le i\le n$.
\end{cor}

\begin{proof}
As $a_i, b_i\in\Z$ for all $i$, we can take $G_n(g)=(F_n(g))^2$.
\end{proof}

\begin{proof}[Proof of Proposition \ref{prop:conjbound}]
Fix the form $q_{6,1}:=\langle 1,1,1,1,1,1,-1\rangle$ on $\mathbb{Q}^7$.
If $q_c$ is any form such that $q_{6,1}$ is $\Q$--isometric to $q_c \oplus q$, then the existence of $D$ is immediate. To compute the upper bound on $D$, we will give upper bounds on the number $S=\lcm_{i,j}\{\denom(P_{ij})\}$, where $P=(P_{ij})$ is a rational matrix representing the isometry which takes $q_c\oplus q$ to $q_{6,1}$.
Given such an $S$, we immediately see that the integral congruence sublattice $L(S^2)=\oplus_{i=1}^7S^2\Z\subset\Z^7$ has the property that $P^TL(S^2)P\subset \Z^7$ and consequently the stabilizer of $L(S^2)$ is a congruence subgroup of $\SO(q_c\oplus q,\Z)$ which gets mapped to a subgroup of $\SO(q_{6,1},\Z)$.
By examining the orders of finite groups of Lie type (see for instance the tables of Ono \cite[Table 1]{Ono}), one can see that the index of such a congruence subgroup is bounded above by $D=S^{42}$.
Therefore finding a bound for $S$ will complete the proof.

To this end, note first by Lemma \ref{lemma:CompFormLemma} that $q_c$ may be chosen so that the product of its coefficients are bounded above by $D_0~\! d^{16}$, where $d=z_1z_2z_3z_4$ and $D_0$ is a constant depending only on $d$.
Writing $\omega=D_0~\! d^{17}$ we see that the coefficients of $q_c\oplus q$ are bounded above by $\omega$.

We now implement repeatedly the algorithm used in Proposition \ref{prop:CongMat} and the bounds in Corollaries \ref{cor:FLemma} and \ref{cor:GLemma} to construct our isometry $P$. For the first bound, we assume that none of the coefficients of $q_c\oplus q$ are $\pm 1$ at any stage of our algorithm, as otherwise Witt cancelation would allow us to improve our bounds. By Proposition \ref{prop:CongMat}, we get a matrix $P^{(1)}$ such that $(P^{(1)})^TA_{q_c\oplus q}P^{(1)}$ represents the diagonal form $g=\langle 1,b_2,\dots,b_6,-b_7\rangle$ of signature $(6,1)$ and with the properties that
$$\lcm_{i,j}\{\denom(P^{(1)}_{ij})\}\le D_1\omega^{4.5},$$
 $\det(P^{(1)})\le D_1'\omega^{306}$, and such that the absolute value of the product of the coefficients of $g$ is bounded above by $D_1''\omega^{612}$.
Running this process $5$ more times and keeping track of the changes in determinant, absolute value of the product of the coefficients, and total denominator change, we end up with a diagonal quadratic form $g'=\langle 1,1,1,1,1,1,-b'_7\rangle$ with $b'_7\in(\Q^*)^2$ and a matrix $P'$ such that $P'^TA_{g}P'=A_{g'}$, $\det(P')\le D_6~\!\omega^{3.101\cdot 10^{12}}$, $b'_7\le D'_6\omega^{6.202\cdot 10^{12}}$, and
$$\lcm_{i,j}\{\denom(P'_{ij})\}\le D''_6\omega^{3.82\cdot 10^{11}}.$$
Let $P''$ be the rational diagonal matrix $P''=\textrm{diag}(1,1,1,1,1,1,1/\sqrt{b_7})$ and let $P=P'P''$, then $P$ is a $\Q$-isometry taking $q_c\oplus q$ to $q_{6,1}$ with the property that
$$S=\lcm_{i,j}\{\denom(P_{ij})\}\le D_7\omega^{3.5\cdot 10^{12}}\le A d^{5.6\cdot 10^{13}},$$
for an absolute constant $A$, which combines the absolute parts of each of the $D_i$.
Therefore $D\le A~\!d^{2.4\cdot 10^{15}}$.

The second part is identical except that we now only need to run the process $5$ total times as opposed to $6$. A similar computation then gives the requisite bound of $S\le A~\! d^{ 1.02\cdot 10^{11}}$ and hence $D\le A~\! d^{ 4.25\cdot 10^{12}}$.
\end{proof}



\subsection{}\label{thebigkahuna}

We deduce the main result of this section and prove Corollary \ref{Larry} using Section \ref{sec:prasadvol}.

\begin{proof}[Proof of Theorem \ref{Sec:ArEx:T1}]
By Theorem \ref{prop:deltaconstruct}, there is $g\in\GL(4,\Q)$ and an integral quadratic form $q'$ such that
$$[g\Gamma g^{-1}:g\Gamma g^{-1} \cap \SO^+(q',\Z)] \leq C_\epsilon V^{\epsilon},$$
for any $\epsilon>0$, where $V$ is the volume of $\mathbb{H}^3/\Gamma$.
Moreover Equation \eqref{eqn:newform} shows that $q'$ is similar to an explicit integral quadratic form which has a $1$ for its first coefficient.
Notice from Equation \eqref{eqn:newform} that $|\disc(q'')|\le d^7$.
By Proposition \ref{prop:conjbound}, there exists a subgroup $\Delta \leq \SO^+(q',\Z)$ of index at most $D$, where $D$ is a constant that depends only on $q'$, such that $\Delta$ admits an injective homomorphism into $\SO^+(6,1;\Z)$.
Therefore taking intersections we conclude that there is a subgroup of index at most $(C_\epsilon D) V^{\epsilon}$ of $\Gamma$ that admits an injective homomorphism into $\SO^+(6,1;\Z)$.
For the explicit bounds on $C_\epsilon$ and $D$, Corollary \ref{cor:indexbound3} implies that $C_\epsilon\le 2^{\epsilon C'_\epsilon+2}11^2d_k^{A_1\omega(d_k)+3/2}$ and Proposition \ref{prop:deltaconstruct} applied to $q''$ gives that $D\le Ad^{2.975\cdot 10^{13}}$.

We note that the injective homomorphism of the finite index subgroup of $\Gamma$ into $\SO^+(6,1;\Z)$ induces a totally geodesic immersion of the associated arithmetic hyperbolic orbifolds.  As noted in the introduction, the second paragraph of Theorem \ref{Sec:ArEx:T1} now follows from \cite[Lem 3.4]{ALR}.
\end{proof}

\begin{cor}\label{Larry}
\LarryCor
\end{cor}

\begin{proof}[Proof of Corollary \ref{Larry}]
To deduce Corollary \ref{Larry}, note that using the results of \cite{Chu} one can circumvent the production of $q_c$ and a $\Q$--isometry $P$ for the Bianchi groups. Indeed \cite[Thm 1.2]{Chu} proves that $\PSL(2,\mathcal{O}_d)$ always contains a special subgroup of index at most $120$ and so the same is true of $\SO^+(q,\Z)$ where $q=\langle 1,1,1,-d\rangle$. Combined with the proof of Theorem \ref{Sec:ArEx:T1}, this shows that any $\Gamma$ commensurable with $\SO^+(q,\Z)$ contains a special subgroup $\Delta$ of index at most $120\, C_\epsilon\, V^\epsilon$ where $C_\epsilon$ depends only on the commensurability class of the Bianchi group $\SO^+(q,\Z)$.
\end{proof}


\subsection{}

For the reader's clarity, we give a couple of concrete examples of $q$ and complimentary form $q_c$, and describe explicitly $C_{1,\epsilon}$, $C_2$, $D$, and the explicit $\Q$--isometry taking $q_c\oplus q$ to $q_{6,1}$ in each case.
The first example is any lattice in the commensurability class of a specific Bianchi group, where the methods of Chu \cite{Chu} already give bounds for this group and its finite index subgroups.
This example is meant to exemplify that, though our bound extends to the entire commensurability class, if one uses the algorithm above then it is many orders of magnitude worse than the uniform bounds produced in \cite{Chu}.
The second example exhibits a new commensurability class to which our techniques apply that is not currently covered in the  literature.

\begin{example}
Let $q=\innp{1,1,1,-7}$ and $\Gamma$ be a fixed lattice commensurable with $\SO^+(q,\Z)$ of covolume at most $V$.
Notice that $\SO^+(q,\Z)$ is a Bianchi group and in particular $\Gamma$ is not cocompact.
As such, via $\Phi$, it is easy to see that the corresponding invariant trace field is $k=\Q(\sqrt{-7})$ and the invariant quaternion algebra is the matrix algebra $\mathcal{A}=(7,7)_{\Q(\sqrt{-7})}\cong\mathrm{M}(2,\Q(\sqrt{-7}))$. We now show how to compute $C_{1,\epsilon}$ when $\epsilon=1/2$ by expanding on each part of Equation \eqref{eqn:covoleqn}.

As $h_k=1$ and $k_\mathcal{A}$ is contained in the narrow class field, we have that $k_\mathcal{A}=k$.
Additionally $\mathcal{A}$ is a matrix algebra so $\Ram_f(\mathcal{A})=\emptyset$ and $r_f=0$.
We can therefore simplify Equation \eqref{eqn:covoleqn} to
$$\frac{7^{3/2}\zeta_k(2)}{8\pi^2}\displaystyle\prod_{\frak{p}\in S}\left(\frac{\Nr(\frak{p})+1}{2}\right)\le  V,$$
which implies that
$$\frac{1}{2^{|S|-2}}\prod_{\frak{p}\in S}\left(\Nr(\frak{p})+1\right)\le 9V.$$
As we are interested in bounding $|S|$ from above we assume that $|S|\ge 2$ which in particular implies that
\begin{equation}\label{eqn:Sprime}
\frac{1}{2^{|S|-2}}\prod_{\frak{p}\in S'}\left(\Nr(\frak{p})+1\right)\le V,
\end{equation}
where $S'$ denotes the set $S$ minus its two smallest norm primes. That implies that $|S'|=|S|-2$ and that no prime contained in $S'$ can divide $2$. As $3,5$ are inert in $k$, Equation \eqref{eqn:Sprime} yields the upper bound $|S|\le 1/2\log_2(V)+2$, which reduces Equation \eqref{eqn:index} to $[\Gamma_{S,\mathcal{O}}:P(\mathcal{E}^1)]\le 8V^{1/2}$. Consequently we may take $C_{1,\epsilon}=8$.

To compute $C_2$ and $D$, we first remark that it is clear from the proof of Proposition \ref{prop:indexbound2} that $C_2=1$ if there is only one conjugacy class of maximal orders, which is the case for our $\mathcal{A}$ since it is a matrix algebra over a PID \cite[Cor 2.2.10]{MR}.
Moreover a complimentary form for $q$ is $q_c=\innp{1,1,7}$ with corresponding $\Q$--isometry
$$P=\begin{pmatrix}
1 & 0 & 0 & 0 & 0 & 0 & 0 \\
0 & 1 & 0 & 0 & 0 & 0 & 0 \\
0 & 0 & 4/7 & 0 & 0 & 0 & 3/7 \\
0 & 0 & 0 & 1 & 0 & 0 & 0 \\
0 & 0 & 0 & 0 & 1 & 0 & 0 \\
0 & 0 & 0 & 0 & 0 & 1 &  \\
0 & 0 & -3/7 & 0 & 0 & 0 & -4/7
\end{pmatrix},$$
to $q_{6,1}$.
Therefore the congruence lattice $L(49)=(49\Z)^7$ of level $49$ is mapped into $\Z^7$ under the isometry $P$ and hence we may take $D\le 49^{42}$.
Putting this all together we see that there is a special subgroup $\Delta<\Gamma$ of index at most $8 (49)^{42} V^{1/2}$.
\end{example}

\begin{example}\label{ex:m306}
Let $\Gamma=\pi_1(M)$ be the fundamental group of the $5/1$ Dehn filling on the manifold $m306$ in SnapPy's closed manifold census \cite{SnapPy}.
We point out that briefly that the manifold $M$ is not fibered, indeed one can check using SnapPy that $M$ is a rational homology sphere.
Considering the upper half plane model of $\mathbb{H}^3$, using SnapPy one can check that $\Gamma$ is a $3$-generated group with holonomy representation given by
\begin{align*}
a&\mapsto\begin{pmatrix}
-\sqrt{1-i}&\frac{\sqrt{2}}{2}(1-i)\\
\frac{-\sqrt{2}}{2}(1-i)&-\sqrt{-1+i}
\end{pmatrix},\\
b&\mapsto\begin{pmatrix}
\sqrt{\frac{1}{2}(-7-5i+\sqrt{4+22i})}&-\sqrt{-4-2i}\\
i-\sqrt{-1+2i}&\sqrt{ \frac{1}{2}(-3+7i+\sqrt{20-10 i}) }
\end{pmatrix},\\
c&\mapsto\begin{pmatrix}
-\sqrt{\frac{1}{2} (-3 + 13 i - \sqrt{-76 + 2 i})}&\sqrt{-10 + 9 i - \sqrt{19 + 62 i}}\\
-\sqrt{-2 + 3 i + \sqrt{-5 - 10 i}}&-\sqrt{\frac{1}{2} (-11 + 5 i + \sqrt{-44 - 62 i})}
\end{pmatrix}.
\end{align*}
Using \cite[Lem 3.5.5]{MR} and \cite[Thm 3.6.1]{MR}, it is straightforward to check that the invariant trace field is $k=k\Gamma=\Q(i)$ and the invariant quaternion algebra has Hilbert symbol
$$\mathcal{A}=A\Gamma= (-8,-20-20i)_{\Q(i)}\cong (2,5)_{\Q(i)}.$$
As $\rho(\gamma)$ has traces which are algebraic integers for all $\gamma\in\Gamma$ (equivalently $\rho(\gamma^2)$ has traces in $\Z[i]$ for all $\gamma\in\Gamma$), \cite[Thm 8.3.2]{MR} shows that $\Gamma$ is in fact arithmetic and therefore under the isomorphism $\Phi$ given in Section \ref{sec:prasadvol}, the image of $\Gamma$ is a lattice commensurable with $\SO^+(q,\Z)$ where $q=\langle 1,2,5,-10\rangle$.
Note that $q$ is anisotropic and has non-trivial Hasse--Witt invariants at primes $2$ and $5$, from which one can see that $\SO^+(q,\Z)$ is not commensurable with any of the lattices contained in \cite{Chu}.

We now compute upper bounds for $C_{1,1} V$, $C_2$, and $D$ explicitly, where we have chosen $\epsilon=1$.
As $\Q(i)$ also has class number one, we again have that $k_\mathcal{A}=k$.
Moreover $|\Ram_f(\mathcal{A})|=2$ and each prime in $\Ram_f(\mathcal{A})$ has norm $5$, consequently Equation \eqref{eqn:covoleqn} simplifies to give 
$$\frac{4\zeta_k(2)}{\pi^2}\displaystyle\prod_{\frak{p}\in S}\left(\frac{\Nr(\frak{p})+1}{2}\right)\le  V.$$
Using SnapPy, one can compute that $\vol(M)=3.66386...$ which combined with the above gives that
$$\displaystyle\prod_{\frak{p}\in S}\left(\frac{\Nr(\frak{p})+1}{2}\right)\le \frac{\pi^2 V}{4\zeta_k(2)}=6.$$
By definition, $S$ must be disjoint from $\Ram_f(\mathcal{A})$ and combining this with an enumeration of the small norm primes in $\Q(i)$, we see that either $S=\emptyset$ or $|S|=1$.
Consequently, $C_{1,1}$ can be chosen so that $C_{1,1} V\le 16$ with $C_{1,1} V$ is as in Proposition \ref{prop:indexbound}.
Moreover, Magma \cite{Magma} shows that the number of conjugacy classes of maximal orders in $\mathcal{A}$, i.e. the type number of $\mathcal{A}$, is $1$ and hence $C_2$ from Proposition \ref{prop:indexbound2} is simply $1$.

To compute an upper bound for $D$ in this setting, note that one complementary form of $q$ is given by $q_c=\langle 2,5,10\rangle$ with corresponding $\Q$--isometry from $q_c\oplus q$ to $q_{6,1}$ given by
$$P=\begin{pmatrix}
1/5 & 0 & -3/10 & 3/4 & 0 & 1/10 & 9/20 \\
-1/5 & 0 & 0 & 0 & 0 & 2/5 & 0 \\
0 & 0 & -9/20 & 9/40 & 11/20 & 0 & 27/40 \\
0 & 1 & 0 & 0 & 0 & 0 & 0 \\
-3/5 & 0 & -1/10 & 1/4 & 0 & -3/10 & 3/20 \\
0 & 0 & -3/5 & 0 & 0 & 0 & 2/5 \\
0 & 0 & -11/20 & 11/40 & 9/20 & 0 & 33/40
\end{pmatrix},$$
Hence we have that $D\le (1600)^{42}$ and $\Gamma$ admits a special subgroup $\Delta$ of index at most $16(1600)^{42}$.
\end{example}


\section{Excluding group elements with right-angled polyhedra}\label{pplus}

We intend to apply our first main result, Theorem \ref{Sec:ArEx:T1}, to produce explicit bounds on the geodesic residual finiteness growth for closed manifolds that satisfy the hypotheses of that theorem. To accomplish this, in \ref{pprrooff} we extend the methods of Patel \cite{Patel} leveraging totally geodesic immersions into right-angled reflection orbifolds. But first, in \ref{geo vs rf}, we recall some definitions and justify an assertion from the introduction.


\subsection{}\label{geo vs rf} For a finitely generated residually finite group $\Gamma$, we define $\mathrm{D}_\Gamma(\gamma)$ to be the minimum of $[\Gamma:\Delta]$ such that $\gamma \notin \Delta$ and $\Delta < \Gamma$. When $\Gamma = \pi_1 M$ for a closed hyperbolic $n$--manifold $M$, two measurements of complexity for the elements of $\Gamma$ can be used to study the extremal behavior of $\mathrm{D}_\Gamma$. First, we have the geodesic length function $\ell(\gamma)$, and second, for a fixed finite generating subset $X$ of $\Gamma$, we have the associated word length $\norm{\gamma}_X$. These can be used to measure the growth rate of the function $\mathrm{D}_\Gamma$. Specifically, we can take the maximum of $\mathrm{D}_\Gamma$ on the finite subsets of $\Gamma$ of non-identity elements $\gamma$ with either $\ell(\gamma) \leq n$ or $\norm{\gamma}_X \leq n$, yielding the \textbf{geodesic residual finiteness growth function} $F_{M,\rho}(n)$ ($\rho$ being the complete hyperbolic metric on $M$) or \textbf{residual finiteness growth function} $F_{\Gamma,X}(n)$, respectively. (Cf.~\cite[\S 2.1]{BRHP} and the introduction to \cite{Patel}.)

Lemma 6.1 of \cite{Patel} asserts that when $M$ is closed, the functions $F_{M,\rho}(n)$ and $F_{\Gamma,X}(n)$ have the same \textbf{asymptotic growth rate}, meaning that there exist real numbers $c,d>0$ for which both
\[ F_{M,\rho}(n) \leq c F_{\Gamma,X}(cn) \quad\mbox{and}\quad F_{\Gamma,X}(n) \leq d F_{M,\rho}(dn). \]
(Note that it is an easy exercise from this definition to show that all linear functions $\mathbb{N}\to\mathbb{N}$ have the same asymptotic growth rate.) The key step in the proof of \cite[Lem 6.1]{Patel} lies in relating $\ell$ to the translation length function $\ell_p(\gamma) = d_{\mathrm{hyp}}(\gamma\cdot \tilde{p},\tilde{p})$ determined by a choice of $p \in M$ and $\tilde{p} \in \widetilde{M}$, since by the \v{S}varc--Milnor Lemma, $\Gamma$ equipped with the norm $\ell_p$ and $\Gamma$ equipped with the norm $\norm{\cdot}_X$ are quasi-isometric.

Theorem 1.1 of \cite{BRHP} asserts that $F_{A_{\Lambda},X}(n) \leq n+1$ for a right-angled Artin group $A_{\Lambda}$ determined by a simplicial graph $\Lambda$, where $X$ is the ``standard'' generating set for $A_{\Lambda}$ (with one generator for each vertex of $\Lambda$). Standard results on residual finiteness growth then imply that every \textbf{virtually special} group $\Gamma$, that is, one with a finite-index subgroup that quasi-isometrically embeds in a right-angled Artin group, has at most linear residual finiteness growth.  This holds in particular when $\Gamma = \pi_1 M$ for a closed hyperbolic $3$-manifold $M$, by \cite{Agol_vHaken}.  In this case it therefore follows from \cite[Lem 6.1]{Patel} that the \textit{geodesic} residual finiteness growth is also at most linear.

When we assert that the geodesic residual finiteness growth of $M$ is \textbf{at most linear}, we mean that there exists a linear function $L\co\mathbb{N}\to\mathbb{N}$ and $c>0$ such that $F_{M,\rho}(n)\leq cL(cn)$ for all $n$, or equivalently, that $F_{M,\rho}(n)\leq c^2L(n)$, since $L$ is linear. Chasing through the definitions, we thus find that there exists a potentially larger constant $K$ such that for every loxodromic element $\alpha\in\pi_1 M$ there is a subgroup $H$ of $\pi_1 M$ with $\alpha\notin H$ and
\[ [\pi_1 M:H] \leq K\,\ell(\alpha), \]
where $\ell(\alpha)$ is the length of the geodesic representative of $\alpha$ in $M$. That is, we obtain equation \eqref{haha}.

This establishes our assertions from the introduction. We emphasize again that the dependence of the constant $K$ on both the generating set $X$ and the minimal index of a special subgroup of $\pi_1 M$ make it difficult to \textit{explicitly} bound geodesic residual finiteness growth by using \cite{BRHP}, so our approach will be different.


\subsection{}\label{pprrooff} We now begin laying the groundwork for the proof of our second main result, Theorem \ref{red sauce}.  The tools that we add to the methods of \cite{Patel} allow us control the interactions between neighborhoods of the ideal points of a finite-volume right-angled polyhedron $P$ in $\mathbb{H}^n$ and a compact hyperbolic manifold immersed totally geodesically in the reflection orbifold determined by $P$.

\begin{lemma}\label{stabilizer} For a right-angled polyhedron $P\subset\mathbb{H}^{n+1}$, an ideal vertex $v$ of $P$, and a horoball $B$ centered at $v$ and embedded in $P$ (in the sense of Definition \ref{embedded}), if $\Gamma_P$ is the group of generated by reflections in the sides of $P$ then for $\gamma\in\Gamma_P$, $B\cap\gamma.B\neq\emptyset$ if and only if $\gamma$ lies in the stabilizer $\Gamma_P(v)$ of $v$ in $\Gamma_P$.\end{lemma}

\begin{proof}  Since $B$ is embedded in $P$, $P\cap\partial B$ is a right-angled polyhedron in $\partial B$, which inherits a Riemannian metric isometric to the Euclidean metric on $\mathbb{R}^{n}$ from $\mathbb{H}^{n+1}$.  Therefore by the Euclidean case of the Poincar\'e polyhedron theorem (see e.g. \cite[Thm 13.5.3]{Ratcliffe}), $\partial B$ is tiled by translates of $P\cap\partial B$ under the action of the group generated by reflections in its sides.  Each such reflection is the restriction to $\partial B$ of the reflection of $\mathbb{H}^n$ in a side of $P$ that contains $v$; in particular, in an element of $\Gamma_P(v)$.  It follows that:
\begin{align}\label{union} B \subset \bigcup \left\{\,\gamma.P:\gamma\in\Gamma_P(v)\,\right\}. \end{align}
Now suppose for some $\gamma\in\Gamma_P$ that $B\cap\gamma.B\ne\emptyset$, and let $x$ be a point in the intersection and $v' = \gamma.v$ be the ideal point of $\gamma.B$.  Applying the above to $B$ and $\gamma.B$ yields $\lambda_0\in\Gamma_P(v)$ and $\lambda_1\in \Gamma_P(v') = \gamma\Gamma_P(v)\gamma^{-1}$ such that $\lambda_0^{-1}.x$ and $\gamma^{-1}\lambda_1^{-1}.x$ lie in $B\cap P$.  Thus $\gamma^{-1}\lambda_1^{-1}\lambda_0$ takes $B\cap P$ to intersect itself.  As $P$ is a fundamental domain for $\Gamma_P$ and $B$ is embedded in $P$ this implies that $\gamma^{-1}\lambda_1^{-1}\lambda_0^{-1}$ is either the identity or the reflection in a side of $P$ containing $v$.  In any case it follows that $\gamma\in\Gamma_P(v)$, since $\lambda_1 = \gamma\lambda_1'\gamma^{-1}$ for some $\lambda_1'\in\Gamma_P(v)$.\end{proof}

Throughout the remainder of this section, we now fix the following standing assumptions which are the same as those in Theorem \ref{red sauce}. Let $n\geq 2$, let $P$ be a right-angled polyhedron in $\mathbb{H}^{n+1}$ with finite volume and at least one ideal vertex, let $\Gamma_P$ be the group generated by reflections in the sides of $P$, and let $\calb$ be a collection of horoballs, one for each ideal vertex of $P$, that are each embedded in the sense of Definition \ref{embedded} and pairwise non-overlapping. 
Moreover, fix a closed hyperbolic $m$--manifold $M$, for $m\leq n$, that admits a totally geodesic immersion $f\co M\to \mathbb{H}^{n+1}/\Gamma_P$.

Then $f$ lifts to a totally geodesic embedding $\tilde{f}$ from the universal cover $\widetilde{M}$ of $M$, which is isometric to $\mathbb{H}^m$, to an $m$--dimensional hyperplane of $\mathbb{H}^{n+1}$.  This map is equivariant with respect to the actions of $\pi_1 M$  and $f_*(\pi_1M)\subset\Gamma_P$ by covering transformations, and since $\tilde{f}$ is a lift its composition with the projection $\mathbb{H}^{n+1}\to\mathbb{H}^{n+1}/\Gamma_P$ equals the composition of the universal cover $\widetilde{M}\to M$ with $f$.

Below we will mostly just identify $\widetilde{M}$ with its image $\tilde{f}(\widetilde{M})$, a totally geodesic copy of $\mathbb{H}^m$ in $\mathbb{H}^{n+1}$, and likewise $\pi_1 M$ with $f_*(\pi_1 M)$, a subgroup of $\Gamma_P$ that stabilizes this copy of $\mathbb{H}^m$ and acts cocompactly on it. This holds except in the statement of the lemma below, where for clarity we highlight the role of the lift $\tilde{f}$.

\begin{lemma}  For each $B\in\calb$ and $\gamma\in\Gamma_P$, if $\gamma.B\cap\tilde{f}(\widetilde{M})$ is non-empty then $\tilde{f}^{-1}(\gamma.B)$ is a compact metric ball in $\widetilde{M}$ with radius $r_h$ satisfying $\cosh r_h = e^{h_{\gamma.B}}$, where $h_{\gamma.B}$ is the maximum, taken over all $x\in\gamma.B\cap\tilde{f}(\widetilde{M})$, of the distance from $x$ to $\partial B$.  The interior of $\tilde{f}^{-1}(B)$ embeds in $M$ under the universal cover $\widetilde{M}\to M$.\end{lemma}

\begin{proof} In the proof we work exclusively in $\mathbb{H}^{n+1}$: we identify $\widetilde{M}$ with $\tilde{f}(\widetilde{M})\subset\mathbb{H}^{n+1}$ and call it $\mathbb{H}^m$; we identify $\pi_1 M$ with $f_*(\pi_1 M)\subset \Gamma_P$, stabilizing $\mathbb{H}^m$; and, for any $\gamma\in\Gamma_P$, we identify $\tilde{f}^{-1}(\gamma.B)$ with $\gamma.B\cap\mathbb{H}^m$. 

The boundary at infinity of $\mathbb{H}^m$ does not contain an ideal point of any $\Gamma_P$-translate of $P$: if it did then $\pi_1M$, which acts preserving the tiling of $\mathbb{H}^m$ by its intersection with such translates, would have a non-compact fundamental domain, contradicting cocompactness.  Since the horoballs $\gamma.B\cap\mathbb{H}^m$ are each centered at such points, for each such $\gamma$, $\mathbb{H}^m$ does not contain the ideal point of $\gamma.B$.

Suppose now that $\mathbb{H}^m$ does intersect $\gamma.B$ for some $\gamma\in\Gamma_P$.  Lemma \ref{stabilizer} implies that for any $\lambda\in\pi_1 M$ that takes $\gamma.B$ to overlap with itself, $\lambda$ lies in the stabilizer $\Gamma_P(\gamma.v)$ of the ideal point $\gamma.v$ of $\gamma.B$.  But all such elements are parabolic, and $\pi_1 M$ has no parabolic elements since it acts cocompactly.  It follows that the interior of $\gamma.B\cap\mathbb{H}^m$ embeds in $M$ under the universal cover.

Working in the Poincar\'e ball model $\mathbb{D}^{n+1}$ for $\mathbb{H}^{n+1}$, we translate $\mathbb{H}^m$ and $\gamma.B$ by isometries so that $\mathbb{H}^m = \mathbb{D}^m\times\{\mathbf{0}\}$ and the ideal point of $\gamma.B$ is at $(0,\hdots,0,1)$.  Then $\gamma.B$ is a Euclidean ball with radius $r\in[1/2,1)$ and Euclidean center $(0,\hdots,0,1-r)$.  By the Pythagorean theorem, $\gamma.B$ therefore intersects $\mathbb{D}^m\times\{\mathbf{0}\}$ in a Euclidean ball of radius $\sqrt{2r-1}$ in $\mathbb{D}^m$, centered at $\mathbf{0}$.  We now recall the formula for the hyperbolic distance $d$ in $\mathbb{D}^n$ (see eg.~\cite[Thm 4.5.1]{Ratcliffe}):\[
	\cosh d(\bx,\by) = 1 + \frac{2|\bx-\by|^2}{(1-|\bx|^2)(1-|\by|^2)}, \]
where $|\cdot|$ is the Euclidean norm.  Therefore the hyperbolic radius $r_h$ of the ball of intersection satisfies $\cosh r_h = r/(1-r)$.  On the other hand, some manipulation shows that the hyperbolic distance $h$ from $\mathbf{0}$ to the lowest point $(0,\hdots,0,1-2r)$ of $B$ satisfies $e^h = r/(1-r) = \cosh r_h$.  And this is the closest point of $\partial B$ to $\mathbf{0}$, since the formula above gives $\cosh d(\mathbf{0},\by) = 1 + 2|\by|^2/(1-|\by|^2)$ for any $\by\in\partial B$.  This increases with $|\by|^2$, which in turn increases with $y_n$, as can be discerned by rearranging the equation $|\by-(0,\hdots,0,1-r)|^2=r^2$ to $|\by|^2 =2r-1 + 2(1-r)y_n$. Given any $\bx = (\bx_0,0)\in\mathbb{H}^m\times\{\mathbf{0}\}$, there is a unique point $\by = (\bx_0,y)\in \partial B$ ``directly below $\bx$'', that is, with $y<0$.  A direct computation now shows that the distance from $\bx$ to $\by$ decreases with $|\bx|^2$, so $\mathbf{0}$ is the furthest point of $\mathbb{H}^m\cap B$ from $\partial B$ and the lemma is proved.\end{proof}

For a horoball $B$ of $\mathbb{H}^n$ and a totally geodesic hyperplane $\mathbb{H}^m \subset\mathbb{H}^n$ that is not incident on the ideal point $v$ of $B$, define the \textit{height} of $\mathbb{H}^m$ with respect to $B$ to be the maximal signed distance from points of $\mathbb{H}^m$ to $\partial B$, where the sign is non-negative for points of $\mathbb{H}^m\cap B$. See below, which pictures two qualitatively different horoball-hyperplane interactions in the upper half-plane model when $m=n=1$.

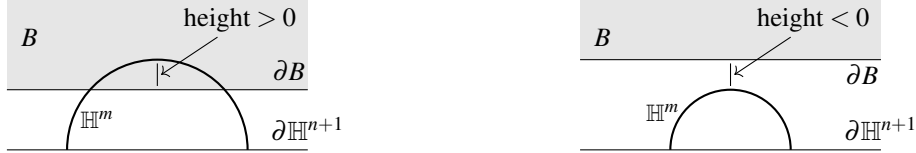
\begin{figure}[ht]
\begin{tikzpicture}

\begin{scope}[xshift=-1.5in]

\draw (-2,0) -- (2,0);
\draw [thick] (1.2,0) arc (0:180:1.2);
\fill [opacity=0.1] (-2,0.8) -- (2,0.8) -- (2,2) -- (-2,2);
\draw (-2,0.8) -- (2,0.8);

\node [above] at (2,0) {$\partial\mathbb{H}^{n+1}$};
\node [above] at (1.75,0.8) {$\partial B$};
\node at (-1.7,1.5) {$B$};
\node at (-0.8,0.4) {\small $\mathbb{H}^m$};

\draw (0,0.85) -- (0,1.15);
\draw [<-] (0.05,1) -- (0.8,1.5);
\node [above] at (1.1,1.45) {height $> 0$};

\end{scope}

\begin{scope}[xshift=1.5in]

\draw (-2,0) -- (2,0);
\draw [thick] (0.8,0) arc (0:180:0.8);
\fill [opacity=0.1] (-2,1.2) -- (2,1.2) -- (2,2) -- (-2,2);
\draw (-2,1.2) -- (2,1.2);

\node [above] at (2,0) {$\partial\mathbb{H}^{n+1}$};
\node [above] at (1.75,0.75) {$\partial B$};
\node at (-1.7,1.5) {$B$};
\node at (-0.9,0.5) {\small $\mathbb{H}^m$};

\draw (0,0.85) -- (0,1.15);
\draw [<-] (0.05,1) -- (0.8,1.5);
\node [above] at (1.1,1.45) {height $< 0$};

\end{scope}

\end{tikzpicture}
\end{figure}

\begin{cor}\label{embedded ball}  For any $B\in\calb$ and $\gamma\in\Gamma_P$, the height $h_{\gamma.B}$ of $\mathbb{H}^m$ with respect to $\gamma.B$ satisfies $e^{h_{\gamma.B}}\leq \cosh r_{\max}$, where $r_{\max}$ is the maximal radius of a ball embedded in $M$.\end{cor}

Below, for a fixed right-angled polyhedron $P\subset\mathbb{H}^{n+1}$ we call the \textit{convexification} of a set $\mathcal{K}\subset\mathbb{H}^{n+1}$ the $P$--convexification from \cite[Defn 2.1]{Patel}: it is the minimal convex union of $\Gamma_P$--translates of $P$ containing $\mathcal{K}$.

\begin{lemma}\label{closest point} For any $\alpha\in \pi_1 M - \{ \mathrm{Id}_{\pi_1 M}\}$, let $\tilde{\alpha}$ be the geodesic axis in $\mathbb{H}^{n+1}$ of $f_*(\alpha)$.  Any polyhedron $P_i$ in the convexification of $\tilde{\alpha}$ intersects the $R$-neighborhood of $\tilde{\alpha}$, where $R = \ln(\sqrt{n+1}+\sqrt{n})$.\end{lemma}

\begin{proof} The proof follows the strategy of Lemmas 3.1 and 4.2 of \cite{Patel}, which respectively establish the cases $n=2$ and $n=3$ (i.e. where $P$ is $3$-- or $4$--dimensional). We point the readers to Figures 1-5 in that paper for the geometric intuition behind this argument. As in those proofs we work in the ball model $\mathbb{D}^{n+1}$ for $\mathbb{H}^{n+1}$ and fix a $\Gamma_P$--translate of $P$ (which we will again just call $P$) that does not intersect the $R$--neighborhood of $\tilde{\alpha}$.  The goal is to show that $\tilde\alpha$ and $P$ are on opposite sides of a hyperplane containing one of the faces of $P$, from which it follows that $P$ is not in the convexification.

We suppose first that the closest point of $P$ to $\tilde\alpha$ is a vertex $e$, and move the entire picture by isometries so that $e$ lies at the origin.  The sides of $P$ that contain $e$ are contained in totally geodesic hyperplanes, each of which is the intersection of a Euclidean hyperplane with $\mathbb{D}^{n+1}$ since it contains the origin.  Their intersections with $\mathbb{S}^n$ divide it into right-angled spherical simplices.  The key computation here is the in-radius of such a simplex; that is, the minimum radius of a metric sphere in $\mathbb{S}^n$ that intersects every hyperplane. 

\begin{claim} An all-right simplex in $\mathbb{S}^n$ has in-radius $\theta=\cos^{-1}\left(\frac{\sqrt{n}}{\sqrt{n+1}}\right)$.\end{claim}

Deferring the claim's proof for the moment, we describe its application to our situation following \cite[Lem 3.1]{Patel}.  Let $j$ be the geodesic hyperplane containing $\tilde\alpha$ that is perpendicular to the arc $\overline{0y}$ from $e$ (which we have moved to $0$) to the closest point $y$ to $e$ on $\alpha$.  The fact that $d(e,\tilde\alpha) > R$ for $R = \ln(\sqrt{n+1}+\sqrt{n})$ ensures that $j$ intersects $\partial\mathbb{D}^{n+1} = \mathbb{S}^n$ in a sphere of radius (in the spherical metric) less than $\cos^{-1}\left(\frac{\sqrt{n}}{\sqrt{n+1}}\right)$, by a calculation entirely analogous to the one spanning pp.~93--94 of \cite{Patel}.  In particular, the ``cross sectional view'' of Figure 3 there still holds (the cross section just has higher codimension).  This sphere is therefore disjoint from the intersection with $\mathbb{S}^n$ of at least one hyperplane containing a side of $P$ that contains $e$. It follows as in \cite{Patel} that this hyperplane separates $\tilde\alpha$ from $P$.

\begin{proof}[Proof of claim] After applying a sequence of orthogonal transformations we may take the given hyperplanes to be the intersections with $\mathbb{S}^n$ of the coordinate planes in $\mathbb{R}^{n+1}$: apply an orthogonal transformation that moves the first hyperplane's normal vector to $\be_1$, then apply an orthogonal transformation of $\be_1^{\perp}$ that moves the second hyperplane's normal vector to $\be_2^{\perp}$, etc.  The coordinate hyperplanes divide $\mathbb{S}^n$ into right-angled simplices, each with the property that for any two of its points, the $i^{\mathrm{th}}$ entry of the first has the same sign as the $i^{\mathrm{th}}$ entry of the second for each $i\in\{1,\hdots,n+1\}$.  We restrict our attention to the simplex $\sigma_n$ consisting of points with all entries non-negative, noting that any of the others is isometric to $\sigma_n$ by a map which simply multiplies each entry by $\pm 1$.

Note that the symmetric group $S_{n+1}$ acts isometrically on $\mathbb{S}^n$ by permuting entries, preserving $\sigma_n$ and acting transitively on its set of faces of dimension $k$, for any fixed $k<n$.  The barycenter of $\sigma_n$, the sole global fixed point in $\sigma_n$ of this action, is $\bv_n = \frac{1}{\sqrt{n+1}}(1,\hdots,1)$.  Similarly call $\bv_k$ the barycenter of $\sigma_k\subset \mathbb{S}^k$ for each $k<n$.  Upon including $\sigma_k$ in $\sigma_n$ by the map $\mathbb{R}^{k+1}\to\mathbb{R}^{k+1}\times\{\mathbf{0}\}\subset\mathbb{R}^n$, we directly compute the spherical distance $d(\bv_n,\bv_k)$ from $\bv_n$ to $\bv_k$ via:
\[ \cos d(\bv_n,\bv_k) = \left[\frac{1}{\sqrt{n+1}}(1,\hdots,1)\right]\cdot \left[\frac{1}{\sqrt{k+1}}(\overbrace{1,\hdots,1}^{k+1},0,\hdots,0)\right]  = \frac{\sqrt{k+1}}{\sqrt{n+1}} \]
It is straightforward to prove that $\bv_k$ is the closest point of $\sigma_k$ to $\sigma_n$.  For each $\bx = (x_1,\hdots,x_{k+1},0,\hdots,0)\in \sigma_k$, $\bx\cdot\bv_n = \bx\cdot\pi(\bv_n)$, where $\pi(\bv_n) = \frac{1}{\sqrt{n+1}}(1,\hdots,1,0,\hdots,0)$ is the projection of $\bv_n$ to $\mathbb{R}^{k+1}\times\{\mathbf{0}\}$.  The Cauchy--Schwarz inequality asserts that $\bx\cdot\pi(\bv_n)\leq \|\bx\|\|\pi(\bv_n)\| = \frac{\sqrt{k+1}}{\sqrt{n+1}}$, with equality holding if and only if $\bx$ is a scalar multiple of $\pi(\bv_n)$. Since the inverse cosine is a decreasing function, the assertion follows.

We note in particular that $d(\bv_n,\bv_k)$ decreases with $k$.  So the closest points to $\bv_n$ on $\partial\sigma_n$, which is a union of $S_{n+1}$--translates of $\sigma_{n-1}$, are the $S_{n+1}$--translates of $\bv_{n-1}$.  Therefore the metric sphere of radius $\cos^{-1}\left(\frac{\sqrt{n}}{\sqrt{n+1}}\right)$ centered at $\bv_n$ is inscribed in $\sigma_n$ and tangent to $\partial\sigma_n$ at each $S_{n+1}$--translate of $\sigma_{n-1}$.  In particular, this sphere intersects every side of $\sigma_n$.

To establish the claim it remains to show for each $\bv\in\sigma_n$ that there is some side of $\sigma_n$ that is at least as far from $\bv$ as from $\bv_n$.  To this point we note that if $\bv = (v_1,\hdots,v_{n+1})\in\sigma_n-\{\be_{n+1}\}$ then the closest point of $\sigma_{n-1}$ to $\bv$ is $\bx = \pi(\bv)/\|\pi(\bv)\|$, where $\pi(\bv) = (v_1,\hdots,v_n)$.  This follows from the Cauchy-Schwarz inequality as above.  We compute that $\pi(\bv)\cdot\pi(\bv) = v_1^{2}+\hdots+v_n^2 = 1 - v_{n+1}^2$, so
\[ d(\bv,\sigma_{n-1}) = d(\bv,\bx) = \cos^{-1}\left(\frac{\bv\cdot\pi(\bv)}{\|\pi(\bv)\|}\right) = \cos^{-1}\sqrt{1-v_{n+1}^2}. \]
(This formula also holds for $\bv = \be_{n+1}$, which has distance $\pi/2 = \cos^{-1}(0)$ from all points of $\sigma_{n-1}$.)  Each other side of $\sigma_n$ is also contained in a coordinate plane; call $\sigma_{n-1}^{(i)}$ the side contained in the coordinate plane perpendicular to $\be_i$ (so $\sigma_{n-1} = \sigma_{n-1}^{(n+1)}$).  For $\bv\in\sigma$ and $1\leq i\leq n+1$, an analogous argument shows that
\[ d(\bv,\sigma_{n-1}^{(i)}) = \cos^{-1}\sqrt{1-v_i^2}. \]
The right side of this equation increases with $v_i$, so for fixed $\bv$ the distance to $\sigma_{n-1}^{(i)}$ is maximized at any $i$ for which $v_i$ is maximal.  But the maximum entry of $\bv$ is at least $1/\sqrt{n+1}$ since $\|\bv\|=1$.\end{proof}

It remains to consider the case when the nearest point of $P$ to $\tilde\alpha$ is not a vertex.  We handle this case by induction, more or less: if the closest point $p$ of $P$ to $\tilde\alpha$ lies in the interior of a face $e$ of codimension $k\leq n$ then we work in the $k$--dimensional geodesic subspace $L$ of $\mathbb{H}^{n+1}$ that contains $p$ and is orthogonal to the $(n+1-k)$--plane containing $e$.  For each side of $P$ that contains $e$, the hyperplane containing it intersects $L$ perpendicularly in a codimension-one geodesic subspace, and the collection of all these subspaces determines a polyhedron in $L$ which contains $P\cap L$ and has a single vertex at $p$.  This polyhedron intersects $\partial L$ in an all-right spherical simplex of dimension $k-1$, which by the claim has in-radius $\cos^{-1}\left(\frac{\sqrt{k-1}}{\sqrt{k}}\right)$.

This quantity is larger than $\cos^{-1}\left(\frac{\sqrt{n}}{\sqrt{n+1}}\right)$, so for $j$ as above it follows that the intersection with $L$ of at least one hyperplane containing a side of $P$ does not intersect $j\cap L$.  Since both $j$ and this hyperplane intersect $L$ orthogonally, it follows that $j$ misses this hyperplane, which hence again separates $\tilde\alpha$ from $P$.\end{proof}

\begin{lemma}  A tubular neighborhood in $\mathbb{H}^{n+1}$ of radius $R$ around a geodesic segment of length $\ell$ has volume $\mathrm{Vol}(\mathbf{B}^n)\sinh^n(R)\ell$, where $\mathrm{Vol}(\mathbf{B}^n)$ is the Euclidean volume of the unit ball in $\mathbb{R}^n$.\end{lemma}

\begin{proof}  This is a straightforward generalization of Lemmas 3.2 and 4.1 of \cite{Patel}.  Details are worked out in the preprint version \cite{Patel_pre} of \cite{Patel}, see Lemma 6.2 there.\end{proof}

\noindent We now prove Theorem \ref{red sauce}, where for the reader's convenience we recall our standing assumptions in the statement.

\begin{thm}\label{red sauce}\RedSauceThm\end{thm}

\begin{proof}  With hypotheses of Theorem \ref{red sauce}, let $\tilde{\alpha}\subset \mathbb{H}^m$ be the geodesic axis of $\alpha$, where $\mathbb{H}^m$ is the totally geodesic subspace of $\mathbb{H}^{n+1}$ stabilized by $\pi_1 M$.  We claim that every polyhedron $\gamma.P$ in the convexification $\calc$ of $\tilde{\alpha}$ has its closest point to $\tilde{\alpha}$ in $\gamma.\mathcal{N}_{R+h_{\max}}$, where $\mathcal{N}_{R+h_{\max}}$ is the $(R+h_{\max})$--neighborhood of $\overline{P-\bigcup\{B\in\calb\}}$.

To prove the claim, suppose that for some $\gamma\in \Gamma_P$ such that $\gamma.P$ is in the convexification of $\tilde{\alpha}$, that the nearest point $x$ of $\gamma.P$ to $\tilde{\alpha}$ lies in $\gamma.B$, for some $B\in\calb$, at distance greater than $R$ from $\partial(\gamma.B)$.  Then the nearest point $y$ on $\tilde{\alpha}$ to $x$ also lies in $\gamma.B$, by Lemma \ref{closest point}.  By Corollary \ref{embedded ball}, $y$ is no further from $\partial(\gamma.B)$ than $h_{\max}$, so $x$ lies no further than $R+h_{\max}$ from $\partial(\gamma.B)$. 

The claim implies for each translate $\gamma.P$ in $\calc$ that all of $\gamma.\mathcal{N}_{R+h_{\max}}$ is contained in the $(R+d_{R+h_{\max}})$--neighborhood of $\tilde{\alpha}$.  We obtain the bound of the theorem by arguing as in the proof of \cite[Thm 3.3]{Patel}. For the sake of brevity, we will not repeat the entirety of that proof here. The idea of the proof is to consider one lift $\overline \alpha$ of $\alpha$ to $\mathbb{H}^m$, which lies along $\tilde \alpha$ and,  using $\mathcal{C}$, produce a fundamental domain $F$ for the action of the desired subgroup $H'$, with the property that at least one endpoint of $\overline \alpha$ is contained in the interior of $F$. This property ensures that $\alpha \notin H'$. We then use the estimates produced above to bound the number of polyhedra in $F$, which in turn gives a bound on the index of $H'$ in $\pi_1M$.  \end{proof}


\section{Explicit constants}\label{constants}

Recall that the right-angled polyhedron $P_6$ of Theorem \ref{Sec:ArEx:T1} is a union of translates of a simplex $\sigma\subset\mathbb{H}^6$ which is a fundamental domain for the action of $\mathrm{SO}(6,1;\mathbb{Z})$.  The Coxeter diagram of $\sigma$ is reproduced in Figure \ref{coxeter} with vertices numbered (compare \cite[Fig~1]{ALR} and \cite[Fig~7.3.4]{Ratcliffe}).  It has a vertex for each side of $\sigma$, with two vertices connected by a single edge if their corresponding sides intersect with an interior angle of $\pi/3$. The sides corresponding to the two vertices connected by the doubled edge intersect with an interior angle of $\pi/4$.  Two vertices are not joined by an edge if the sides they represent intersect at right angles.

\begin{figure}[H]
\begin{tikzpicture}[scale=0.8]

\fill (0,0) circle [radius=0.1]; \node [above] at (0,0) {$1$};
\fill (1,0) circle [radius=0.1]; \node [above] at (1,0) {$2$};
\fill (2,0) circle [radius=0.1]; \node [above] at (2,0) {$4$};
\fill (3,0) circle [radius=0.1]; \node [above] at (3,0) {$5$};
\fill (4,0) circle [radius=0.1]; \node [above] at (4,0) {$6$};
\fill (5,0) circle [radius=0.1]; \node [above] at (5,0) {$7$};
\fill (2,-1) circle [radius=0.1]; \node [left] at (2,-1) {$3$};
\draw [very thick] (0,0) -- (4,0);
\draw [very thick] (4,0.07) -- (5,0.07);
\draw [very thick] (4,-0.07) -- (5,-0.07);
\draw [very thick] (2,0) -- (2,-1);


\end{tikzpicture}
\caption{The Coxeter diagram of a simplex $\sigma\subset\mathbb{H}^6$.}
\label{coxeter}
\end{figure}
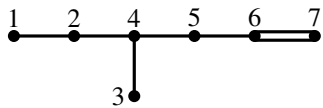

Our first goal here is to understand the geometry of $\sigma$ better.  We will follow the proof of Theorems 7.2.4 and 7.3.1 of \cite{Ratcliffe}, which construct Coxeter simplices, to give an explicit description of $\sigma$ in the hyperboloid model for $\mathbb{H}^6$ (see eg.~\cite[Ch 3]{Ratcliffe} for an introduction to this model).  For each $i$ between $1$ and $7$, let $S_i$ be the side of $\sigma$ corresponding to the vertex labeled $i$.  We will first locate the inward-pointing normal $\bv_i$ to $S_i$ for each such $i$.  Then for each $i$ we will locate the vertex $\bx_i$ of $\sigma$ opposite $S_i$.  (We are following Ratcliffe's notation as closely as possible here; note in particular that $\bv_i$ is \textit{not} a vertex of $\sigma$.)

The Gram matrix $A$ of $\sigma$ can be read off from the Coxeter diagram.  Its $(i,j)$-entry is $-\cos\theta_{ij}$, where $\theta_{ij}$ is the interior angle of $\sigma$ at $S_i\cap S_j$.\[
	A = \begin{pmatrix} 1 & -1/2 & 0 & 0 & 0 & 0 & 0 \\ -1/2 & 1 & 0 & -1/2 & 0 & 0 & 0 \\ 
		0 & 0 & 1 & -1/2 & 0 & 0 & 0 \\ 0 & -1/2 & -1/2 & 1 & -1/2 & 0 & 0 \\ 
		0 & 0 & 0 & -1/2 & 1 & -1/2 & 0 \\ 0 & 0 & 0 & 0 & -1/2 & 1 & -1/\sqrt{2} \\ 
		0 & 0 & 0 & 0 & 0 & -1/\sqrt{2} & 1 \end{pmatrix}. \]
Applying the Gram--Schmidt process to the standard basis of $\mathbb{R}^7$ yields one which is orthonormal with respect to the bilinear form determined by $A$.  A bit more manipulation gives a matrix $C$ with the property that $C^t A C = J$, where $J$ is the diagonal matrix with $(i,i)$-entry equal to $1$ for $i<7$ and $-1$ for $i=7$. \[
	C = \begin{pmatrix} 1 & \frac{-1}{2} & 0 & 0 & 0 & 0 & 0 \\
		0 & \frac{\sqrt{3}}{2} & 0 & \frac{-1}{\sqrt{3}} & 0 & 0 & 0 \\
		0 & 0 & 1 & \frac{-1}{2} & 0 & 0 & 0 \\ 0 & 0 & 0 & \frac{1}{2}\sqrt{\frac{5}{3}} & - \sqrt{\frac{3}{5}} \\
		0 & 0 & 0 & 0 & \sqrt{\frac{2}{5}} & \frac{-1}{2}\sqrt{\frac{5}{2}} & 0 \\
		0 & 0 & 0 & 0 & 0 & \frac{1}{2}\sqrt{\frac{3}{2}} & \frac{-2}{\sqrt{3}} \\
		0 & 0 & 0 & 0 & 0 & 0 & \frac{1}{\sqrt{3}} \end{pmatrix}. \]
(That $C^tAC=J$ can easily be checked with a computer algebra system.)  As in the proof of \cite[Thm 7.2.4]{Ratcliffe}, for each $i$ between $1$ and $7$ the $i$th column $\bv_i$ of $C$ is the inward-pointing normal to the face $S_i$ of $\sigma$, which is itself the intersection with $\mathbb{H}^6$ of the image of the non-negative orthant $\{(x_1,\hdots,x_7)\,|\, x_i\geq 0\}$ under the inverse of the linear transformation determined by $C^tJ$.

For each $i$, the vertex $\bx_i$ of $\sigma$ opposite $S_i$ is the intersection of the faces $S_j$ for $j\neq i$.  It is therefore characterized by the property that $\bx_i\circ \bv_j = 0$, $j\neq i$, where ``$\circ$'' refers to the Lorentzian inner product on $\mathbb{R}^7$.  A little linear algebra therefore yields the following descriptions for the $\bx_i$:\[\begin{array}{rl}
	\bx_7 & = (0,0,0,0,0,0,1), \\
	\bx_6 & = \left(0,0,0,0,0,\frac{-1}{\sqrt{3}},\frac{2}{\sqrt{3}}\right), \\
	\bx_5 & = \left(0,0,0,0,\frac{-1}{2},\frac{-1}{2}\sqrt{\frac{5}{3}},\sqrt{\frac{5}{3}}\right), \\
	\bx_4 & = \left(0,0,0,\frac{-1}{\sqrt{5}},-\sqrt{\frac{3}{10}},\frac{-1}{\sqrt{2}},\sqrt{2}\right), \\
	\bx_3 & = \left(0,0,\frac{-1}{\sqrt{2}},-\sqrt{\frac{3}{10}},\frac{-3}{2\sqrt{5}},\frac{-\sqrt{3}}{2},\sqrt{3}\right), \\
	\bx_2 & = \left(0,\frac{-1}{\sqrt{3}},0,\frac{-2}{\sqrt{15}},-\sqrt{\frac{2}{5}},-\sqrt{\frac{2}{3}},2\sqrt{\frac{2}{3}}\right), \\
	\bx_1 & = \left(-1,\frac{-1}{\sqrt{3}},0,\frac{-2}{\sqrt{15}},-\sqrt{\frac{2}{5}},-\sqrt{\frac{2}{3}},2\sqrt{\frac{2}{3}}\right)\cdot t.
\end{array} \]
Note that $\bx_1$ depends on a real parameter $t$: this is because it does not lie in $\mathbb{H}^6$ but is a line in the light cone representing the sole ideal vertex of $\sigma$.  

\begin{remark} As a check for the computation to this point, we compare with Everitt--Ratcliffe--Tschantz \cite{EvRaTsch}, which also identifies the vertices of a simplex isometric to $\sigma$. (It is called $\Delta^6$ there.) The matrix
\[ \begin{pmatrix} -\frac{1}{\sqrt{2}} & \frac{1}{\sqrt{2}} & 0 & 0 & 0 & 0 & 0 \\ -\frac{1}{\sqrt{6}} & -\frac{1}{\sqrt{6}} & \sqrt{\frac{2}{3}} & 0 & 0 & 0 & 0 \\ \frac{1}{\sqrt{2}} & \frac{1}{\sqrt{2}} & \frac{1}{\sqrt{2}} & 0 & 0 & 0 &-\frac{1}{\sqrt{2}} \\ \frac{1}{\sqrt{30}} & \frac{1}{\sqrt{30}} & \frac{1}{\sqrt{30}} & \sqrt{\frac{6}{5}} & 0 & 0 & -\sqrt{\frac{3}{10}} \\ \frac{1}{2\sqrt{5}} & \frac{1}{2\sqrt{5}} & \frac{1}{2\sqrt{5}} & \frac{1}{2\sqrt{5}} & \frac{\sqrt{5}}{2} & 0 & \frac{-3}{2\sqrt{5}} \\ \frac{1}{2\sqrt{3}} & \frac{1}{2\sqrt{3}} & \frac{1}{2\sqrt{3}} & \frac{1}{2\sqrt{3}} & \frac{1}{2\sqrt{3}} & \frac{2}{\sqrt{3}} & -\frac{\sqrt{3}}{2} \\ -\frac{1}{\sqrt{3}} & -\frac{1}{\sqrt{3}} & -\frac{1}{\sqrt{3}} & -\frac{1}{\sqrt{3}} & -\frac{1}{\sqrt{3}} & -\frac{1}{\sqrt{3}} & \sqrt{3} \end{pmatrix} \in \mathrm{SO}^+(6,1), \]
takes each vertex of $\Delta^6$ listed in Table 1 of \cite{EvRaTsch} to one of the $\bx_i$ described above.  In particular, its product with  $(1,0,0,0,0,0,1)$ is $\bx_1$ (with $t=1/\sqrt{2}$).\end{remark}

Each fixed $t>0$ determines a horoball of $\mathbb{H}^6$ centered at $\bx_1$: the set of points $\by\in\mathbb{H}^6$ satisfying $\by\circ\bx_1 \geq -1$.  (This perspective was exploited by eg.~Epstein--Penner \cite{EpstePen}.)  Direct computation shows that $\bx_2\circ\bx_1 = -t$ is the largest value among the $\bx_j\circ\bx_1$, $2\leq j\leq 7$, for any fixed $t>0$.  Therefore fixing $t=1$ and calling the corresponding horoball $B$, we have that $\bx_2$ lies on the boundary of $B$, with $\bx_j$ outside $B$ for all $j>2$.

Below we summarize the development above, and some additional observations.

\begin{lemma}\label{mabel}  Let $\sigma\subset\mathbb{H}^6$ be the (generalized) hyperbolic simplex with Coxeter diagram given in Figure \ref{coxeter}.  For $1\leq i\leq 7$, let $S_i$ be the side of $\sigma$ corresponding to the vertex labeled $i$ in the figure, and let $\bx_i$ be the vertex of $\sigma$ opposite $S_i$. 

Among the $\bx_i$, $1\leq i\leq 7$, only $\bx_1$ is ideal.  Let $B$ be the horoball of $\mathbb{H}^6$ centered at $\bx_1$  which has $\bx_2$ in its boundary. The totally geodesic hyperplane of $\mathbb{H}^6$ containing $S_1$ intersects $B$ only at $\bx_2$.  In the Euclidean metric that $\partial B$ inherits from $\mathbb{H}^6$, $\sigma\cap\partial B$ is a simplex with volume $1/(2^{9.5}\cdot5\cdot3)$.  Finally, for $d_{\max} = \cosh^{-1}(\sqrt{3})$, the closed $d_{\max}$--neighborhood of $\bx_7$ contains all of $\overline{\sigma-(\sigma\cap B)}$.\end{lemma}

\begin{proof}  The subspace $V_1 = \{0\}\times\mathbb{R}^6$ of $\mathbb{R}^7$ intersects $\mathbb{H}^6$ in the totally geodesic hyperplane $H_1$ containing the face $S_1$ (and hence also $\bx_2,\hdots,\bx_7$ in particular): note that $V_1$ is clearly Lorentz-orthogonal to the first column $\bv_1$ of the matrix $C$.  For any $\by\in H_1$ we have 
\[ \by\circ\bx_1 = \by\circ\bx_2 \leq -1, \]
with equality if and only if $\by=\bx_2$.  Here the equality follows from the explicit descriptions of $\bx_1$ and $\bx_2$ and the fact that $\by$ has first entry equal to zero.  The inequality above follows from a consequence of the Cauchy--Schwarz inequality: if $\bx\circ\bx = a\leq 0$ and $\by\circ\by=b\leq 0$, and $\bx$ and $\by$ have positive $n$th entries, then $\bx\circ\by\leq -\sqrt{ab}$, with equality if and only if $\by$ is a scalar multiple of $\bx$.  We thus find that $H_1\cap B = \bx_2$.

Let $\sigma' = \sigma\cap\partial B$, and for each $i>1$ let $S_i' = S_i\cap\partial B$.  Each such $S_i'$ is a Euclidean hyperplane in the metric that $\partial B$ inherits from $\mathbb{H}^6$, and the angle of intersection between $S_i'$ and $S_j'$ matches that of $S_i$ and $S_j$.  It follows that the Coxeter diagram of $\sigma'$ is obtained from the one in Figure \ref{coxeter} by removing the vertex labeled $1$ and the interior of the edge attached to it.

We now briefly recap the standard fact that $\sigma'$ is the double of a simplex $\sigma_0$ which is a fundamental domain for the symmetries of a five-dimensional Euclidean cube.  The cube is regular; that is, its symmetry group acts transitively on \textit{flags}, tuples of the form $(F_0,F_1,F_2,F_3,F_4,F_5)$ where $F_5$ is the cube and $F_i$ is a codimension-one face of $F_{i+1}$ for each $i<5$.  For instance, taking $F_i = [0,1]^i\times\{\mathbf{0}_{5-i}\}$ for $0\leq i\leq 5$ yields a flag of the cube $[0,1]^5$.  (Here for any $j>0$ and $r\in\mathbb{R}$, ``$\mathbf{r}_j$'' means the vector in $\mathbb{R}^j$ with all entries $r$.)  

We associate a simplex to such a flag by placing a vertex at the \textit{barycenter} of each $F_i$, the point fixed by all symmetries preserving $F_i$.  Vertices associated to the sample flag above are of the form $\by_i = \mathbf{\frac{1}{2}}_i\times\mathbf{0}_{5-i}$ for $0\leq i\leq 5$.  The cube is thus tiled by these simplices, which all have a vertex at its barycenter.  The cube's symmetry group acts transitively on the simplices, so each is a copy of $\sigma_0$.  

Below is the Coxeter diagram of the reflection group in the sides of $\sigma_0$:

\begin{figure}[ht]
\begin{tikzpicture}[scale=0.8]

\fill (0,0) circle [radius=0.1];
\fill (1,0) circle [radius=0.1];
\fill (2,0) circle [radius=0.1];
\fill (3,0) circle [radius=0.1];
\fill (4,0) circle [radius=0.1];
\fill (5,0) circle [radius=0.1];
\draw [very thick] (1,0) -- (4,0);
\draw [very thick] (4,0.07) -- (5,0.07);
\draw [very thick] (4,-0.07) -- (5,-0.07);
\draw [very thick] (1,0.07) -- (0,0.07);
\draw [very thick] (1,-0.07) -- (0,-0.07);

\end{tikzpicture}
\end{figure}

This can be easily checked by an explicit calculation using the sample copy of $\sigma_0$ described above.  From such a calculation one finds that the face $T$ opposite the vertex $F_0$ corresponds to one of the endpoints of the diagram.  That is, $T$ is perpendicular to all other faces save one, which it intersects at an angle of $\pi/4$.  Doubling $\sigma_0$ across $T$ thus yields another simplex $\sigma'$ which has four faces that are doubles of certain faces of $\sigma_0$ --- those perpendicular to $T$.  These faces have the same angles of intersection in $\sigma'$ as in $\sigma_0$.

The remaining two faces of $\sigma'$ are the face $S_2'$ of $\sigma_0$ that meets $T$ at an angle of $\pi/4$ and its image $S_3'$ under reflection across $T$.  These faces are thus perpendicular, and $S_3'$ meets every other face at the same angle as $S_2'$.  In particular they meet a common face $S_4'$ at an angle of $\pi/3$ and all others at right angles.  It follows that the Coxeter diagram of $\sigma'$ is obtained from that of Figure \ref{coxeter} by removing the vertex labeled $1$ and the interior of the edge attached to it, as claimed above.  Moreover, the faces labeled $S_2'$, $S_3'$ and $S_4'$ here play the same roles as the $S_i' = S_i\cap\partial B$ above.

This last observation can be combined with information about the vertices of our particular embedding of $\sigma'$ to discern the edge lengths of the ambient cube.  Note that the vertex of $\sigma_0$ opposite $T$ is also the vertex of $\sigma'$ opposite $S_3'$, since $T$ separates them.  Similarly, the reflection of this vertex across $T$ is opposite $S_2'$ in $\sigma'$.  And the vertex of $\sigma_0$ opposite $T$ is $F_0$, a vertex of the ambient cube, whence also its reflected image is a vertex of the cube, and the two vertices share an edge. On the other hand, in our embedding of $\sigma'$, its vertices opposite $S_2'$ and $S_3'$ are the orthogonal projections $\bx_2'$ and $\bx_3'$ of $\bx_2$ and $\bx_3$, respectively, to $\partial B$.  Since $\bx_2\in\partial B$ we have $\bx_2' = \bx_2$.  The projection of $\bx_3$ to $\partial B$ is along the geodesic ray
\[ \gamma(t) = e^{-t}\bx_3 - \left(\frac{\sinh t}{\bx_3\circ\bx_1}\right)\bx_1,\quad t\geq 0. \]
(One can verify directly that this is a geodesic ray in $\mathbb{H}^6$, parametrized by arclength, that starts at $\bx_3$ and projectively approaches the class of $\bx_1$ as $t\to\infty$.)  Its intersection with $\partial B$ occurs at $t = \ln(-\bx_3\circ\bx_1) = \ln\sqrt{2}$, so $\bx_3' = \gamma(\ln\sqrt{2}) = \frac{1}{\sqrt{2}}\bx_3+\frac{1}{4}\bx_1$.  The hyperbolic distance $d$ from $\bx_2'$ to $\bx_3'$ satisfies 
\[ \cosh d = -\bx_2'\circ\bx_3' = -\frac{1}{\sqrt{2}}\bx_2\circ\bx_3 - \frac{1}{4}\bx_2\circ\bx_1 = \frac{5}{4}, \]
since $\bx_2\circ\bx_3 = -\sqrt{2}$ and $\bx_2\circ\bx_1 = 1$.  Using the fact that the Euclidean distance $\ell$ from $\bx_2'$ to $\bx_3'$ in $\partial B$ satisfies $\ell/2 = \sinh (d/2)$ we obtain $\ell = 1/\sqrt{2}$.  This is thus the sidelength of the ambient Euclidean cube.

Because an $n$--dimensional cube has $2n$ faces, the $5$-cube has $2^5\cdot 5! = 2^8\cdot5\cdot 3$ flags, so it is tiled by this is the number of copies of $\sigma_0$.  Since $\sigma'$ is the double of $\sigma_0$, the ratio of its volume to that of the ambient cube is $1$ to $2^7\cdot 5\cdot 3$.  And since the cube itself has edgelength $1/\sqrt{2}$ and therefore Euclidean volume $1/2^{2.5}$ we obtain the claimed volume for $\sigma'$.

We finally address the claim regarding $d_{\max} = \cosh^{-1}(\sqrt{3})$.  Suppose $\bp = \sum_{i=1}^7 t_i\bx_i$ is an element of $\sigma$ outside the interior of $B$, so
\[ \bp\circ\bx_1 = \sum_{i=2}^7 t_i\bx_i\circ\bx_1 \leq -1\quad\Rightarrow\quad \sum_{i=2}^7 t_i(-\bx_i\circ\bx_1) \geq 1. \]
(Recall $\bx_i\circ\bx_1 < 0$ for each $i$, since $\bx_i\circ\bx_i=-1$ and $\bx_1\circ\bx_1 = 0$.)  That $\bp$ lies in $\sigma$ means $t_i\geq 0$ for all $i$ and
\[ \bp\circ\bp = \bp_0\circ\bp_0 + 2t_1\sum_{i=2}^7 t_i\bx_i\circ\bx_1 = -1. \]
Here $\bp_0 = \sum_{i=2}^7 t_i\bx_i$.  Solving for $t_1$ yields
\[ t_1 = \frac{-1-\bp_0\circ\bp_0}{2\sum_{i=2}^7 t_i\bx_i\circ\bx_1} \leq \frac{1}{2}(1+\bp_0\circ\bp_0). \]
Note that since $t_1$ is non-negative we must have $\bp_0\circ\bp_0\geq-1$.  We now observe that for each $i$, the inner product $\bx_7\circ\bx_i$ is the opposite of the final entry of $\bx_i$.  The least of these quantities is $\bx_7\circ\bx_3 = -\sqrt{3}$.  So we immediately obtain the inequality $\bp\circ\bx_7 \geq -\sqrt{3}\sum_{i=1}^7 t_i$.  Since $\bx_i\circ\bx_i=-1$ for each $i>1$, we have
\[ \bp_0\circ\bp_0 = -\sum_{i=2}^7 t_i^2 + 2\sum_{i\neq j} t_it_j \bx_i\circ\bx_j \leq -\left(\sum_{i=2}^7 t_i\right)^2. \]
Therefore $\bp\circ\bx_7 \geq -\sqrt{3}(t_1 + \sqrt{-\bp_0\circ\bp_0})\geq -\frac{\sqrt{3}}{2}(1+2\sqrt{-\bp_0\circ\bp_0}+\bp_0\circ\bp_0)$.  A calculus argument shows that this is at least $\sqrt{-3}$ regardless of the value of $\bp_0\circ\bp_0$ in $[-1,0]$.  This proves that $d(\bp,\bx_7)\leq d_{\max}$, since their distance is defined as the inverse hyperbolic cosine of $-\bp\circ\bx_7$.\end{proof}

\begin{cor}\label{able} Let $\sigma\subset \mathbb{H}^6$ be the generalized hyperbolic simplex with Coxeter diagram given in Figure \ref{coxeter}, and let $G$ be the group generated by reflections in the sides of $\sigma$ corresponding to vertices $1$ through $6$.  Then $P = \bigcup \{g(\sigma)\,|\,g\in G\}$ is a right-angled polyhedron of finite volume, and for $B$ as in Lemma \ref{mabel}, $\mathcal{B} = \{g(B)\,|\,g\in G\}$ is a collection of horoballs that are embedded in the sense of Definition \ref{embedded} and pairwise non-overlapping, with one for each ideal vertex of $P$.  For $d_{\max}$ as in Lemma \ref{mabel}, $\overline{P-\bigcup\{B\in\calb\}}$ is contained in the closed ball of radius $d_{\max}$ about $\bx_7$, and it has volume 
\[ 2^7\cdot 3^4\cdot 5 \left(\frac{\pi^3}{2^7\cdot5^2\cdot3^5} - \frac{1}{2^{9.5}\cdot5^2\cdot3}\right) = \frac{2^{2.5}\pi^3-3^4}{2^{2.5}\cdot 5\cdot 3} \approx 1.112. \]
\end{cor}

\begin{proof}  That $P$ is a right-angled polyhedron follows from the fact that its face $S_7$ corresponding to vertex $7$ intersects every other face at an angle of $\pi/2$ or $\pi/4$, see \cite[Lem 3.4]{ALR}.  Let $H_0$ be the subgroup of $G$ generated by reflections in the faces $S_2$ through $S_6$ of $\sigma$, and let $P_0 = \bigcup \{h(\sigma)\,|\,h\in H_0\}$.  Then $P$ is a non-overlapping union of translates of $P_0$, one for each (say, left) coset of $H_0$.  By construction, each side of $P_0$ is either a union of $H_0$--translates of $S_1$ or of $S_7$.  The sides of the former kind comprise the frontier of $P_0$ in $P$; those of the latter lie in the frontier of $P$.

We claim that $P\cap B = P_0\cap B$.  By Lemma \ref{mabel}, $B$ is contained in the half-space bounded by the geodesic hyperplane $\calh_1$ containing $S_1$ that also contains $\sigma$. Since each of $S_2$ through $S_6$ contains the ideal vertex $\bx_1$, $H_0$ stabilizes $B$, so each $H_0$--translate of $\calh_1$ bounds a half-space containing both $B$ and the corresponding translate of $\sigma$.  If $\calh_1,\hdots,\calh_n$ is the list of such translates containing a side of $P_0$, then both $B$ and $P_0$ are contained in an intersection of half-spaces bounded by the $\calh_i$.  Therefore since the frontier of $P_0$ in $P$ is a union of $H_0$--translates of $S_1$, each point of $P-P_0$ is separated from $P_0$ by some $\calh_i$.  This proves the claim.

The claim implies that $\mathcal{B} = \{ g(B)\,|\,g\in G\}$ is embedded and pairwise non-overlapping: $\mathcal{B}$ corresponds bijectively to the set of cosets of $H_0$ in $G$, and the intersection of each element with $P$ is contained in a corresponding translate of $P_0$.

The remaining claims follow from the fact that $\overline{P-\bigcup\{B\in\mathcal{B}\}}$ is a union of $G$-translates of $\overline{\sigma-(\sigma\cap B)}$, where $G$ is a group of isometries fixing $\bx_7$.  This and the final claim of Lemma \ref{mabel} immediately imply that $\overline{P-\bigcup\{B\in\mathcal{B}\}}$ is contained in the ball of radius $d_{\max}$ about $\bx_7$.  For the volume, we appeal to \cite{JKRT}, which asserts that $\sigma$ has volume $\pi^3/777600$ (see p.~344 there).  The volume of $\sigma'=\sigma\cap\partial B$ is recorded in Lemma \ref{mabel}, and the volume of $\sigma\cap B$ is one-fifth this quantity.  (This follows from a general fact that can be proven using horoballs centered at infinity in the upper half-space model $\{(x_1,\hdots,x_n)\,|\,x_n>0\}$ for $\mathbb{H}^n$, where the hyperbolic volume form is the Euclidean volume form scaled by $1/x_n^n$.)  Subtracting one from the other, and multiplying the result by the order of $G$, gives the formula claimed.\end{proof}

\begin{cor}\label{effective}\EffectiveCor\end{cor}

\begin{proof}  Fix $\alpha\in\Gamma-\{1\}$ and $\epsilon>0$.  By Theorem \ref{Sec:ArEx:T1}, $\Gamma$ has a subgroup $\Delta$ that injects to $\mathrm{SO}(6,1;\mathbb{Z})$, with index at most $C_{\epsilon}D\mathrm{vol}(M)$, for $C_{\epsilon}$ and $D$ as described in Theorem \ref{prop:deltaconstruct} and \ref{prop:conjbound}, respectively.  By the discussion above, $D\leq Ad^{2.975\cdot 10^{13}}$. So $\Delta$ has index at most $C_{\epsilon}Ad^{2.975\cdot 10^{13}}\mathrm{vol}(M)^\epsilon$, and if $\alpha\notin\Delta$ then we are done.  So we now assume that it is.

Since $P$ is the union of $2^73^45$ copies of $\sigma$, the reflection group $\Gamma_P$ in its sides has that index in the reflection group $\mathrm{SO}(6,1;\mathbb{Z})$ in the sides of $\sigma$.  Therefore $\Gamma_P\cap\Delta$ has index at most $ 2^73^45 \cdot C_{\epsilon}Ad^{2.975\cdot 10^{13}}\mathrm{vol}(M)^\epsilon$ in $\Gamma$.  If $\alpha\in\Delta-\{1\}\subset\mathrm{SO}(6,1;\mathbb{Z})$ is not in $\Gamma_P$ then again we are done, so we now suppose that $\alpha\subset\Gamma_P$.  We will finally apply Theorem \ref{red sauce} to obtain the stated bound.

The remaining constants in the Corollary's statement are obtained by specializing those of Theorem \ref{red sauce} to our example.  For instance, the general formula $v_n(1) = \pi^{n/2}/\Gamma(\frac{n}{2}+1)$ takes the value $8\pi^2/15$ when $n+1=6$.  And the volume $V_0$ of $\overline{P-\bigcup\{B\in\mathcal{B}\}}$ is less than $V_{R+h_{\max}}$.

The polynomial $p$ above arises from the computation of the volume $V_6(r)$ of a ball in $\mathbb{H}^6$ of radius $r$:
\[ V_6(r) = \pi^3\int_0^r \sinh^5 t \mathit{dt} = \pi^3p(\cosh r). \]
This is used to bound $h_{\max}$ above in terms of the volume of $M$. Corollary \ref{embedded ball} implies that $h_{\max}$ is at most $\ln\cosh R$, where $R$ is the radius of the largest ball embedded in $M$.  For $p(x)$ as above we have $\cosh R \leq p^{-1}(\mathrm{vol}(M))$, so $h_{\max}\leq \ln p^{-1}(\mathrm{vol}(M))$.  Since a ball about $\bx_7$ of radius $d_{\max}$ contains all of $P - \bigcup \{B\in\calb\}$ by Corollary \ref{able}, we may bound $d_{R+h_{\max}}$ by twice the radius $R+d_{\max}+h_{\max}$ of a ball at $v$.
\end{proof}

By combining Corollary \ref{effective} with the results of Section \ref{arithmetic bounds}, we are finally in position to compute an explicit value for the constant $K$ appearing in the inequality \eqref{haha}, when $M$ satisfies the Corollary's hypotheses.

\begin{example}\label{m306 taketwo} Recall from  Example \ref{ex:m306} that the closed hyperbolic $3$-manifold $M$ obtained from the census manifold m306 by $5/1$ Dehn filling is commensurable with $\mathrm{SO}^+(q,\mathbb{Z})$ for $q = \langle 1,2,5,-10\rangle$, and for $C_{\epsilon}$ and $D$ as in Theorem \ref{Sec:ArEx:T1} we may take $C_{\epsilon} = 16$ (with $\epsilon =1$) and $D = 1600^{42}$.  Therefore by Corollary \ref{effective}, for this manifold the constant $K$ appearing in Equation \eqref{haha} is
\[  2^73^45 \cdot 16(1600)^{42} \cdot \frac{8\pi^22^{2.5}}{2^{2.5}\pi^3-3^4}\sinh^5\left(2(2\ln(\sqrt{6}+\sqrt{7})+\cosh^{-1}(\sqrt{3})+\ln p^{-1}(4G_{\mathrm{cat}})))\right), \]
where $G_{\mathrm{cat}}$ is Catalan's constant (the volume of this specific $M$ is $4G_\mathrm{cat}$). This is approximately $7\cdot 10^{150}$.\end{example}


\bibliographystyle{abbrv}
\bibliography{RF}


\end{document}